\documentclass[11pt]{amsart}
\usepackage{amssymb, amsfonts, latexsym, amsthm, amsmath, verbatim, enumerate}
\pdfoutput=1
\usepackage{soul} 
\usepackage{paralist}
\usepackage[all]{xy}
\usepackage{amscd}
\usepackage{mathrsfs} 
\usepackage[english]{babel}
\usepackage{tikz}
\usepackage{graphicx}
\usepackage{epstopdf}
\definecolor{darkblue}{rgb}{0,0,0.4} 
\usepackage[colorlinks=true, citecolor=darkblue, filecolor=darkblue, linkcolor=darkblue,urlcolor=darkblue]{hyperref}
\usetikzlibrary{arrows}
\usetikzlibrary{decorations.pathmorphing}
\usepackage{amscd}
\usepackage{enumitem}
\usepackage[color=blue!20!white]{todonotes}
\makeatletter
\providecommand\@dotsep{5}
\def\listtodoname{List of Todos}
\def\listoftodos{\@starttoc{tdo}\listtodoname}
\makeatother

\tikzstyle{crossing}=[circle,fill=white,minimum height=6pt,inner sep=0pt, outer sep=0pt, style={transform shape=false}]

\numberwithin{equation}{section}

\theoremstyle{plain}
\newtheorem{thm}[equation]{Theorem}

\newtheorem*{thm*}{Theorem}

\newtheorem{cor}[equation]{Corollary}
\newtheorem{lem}[equation]{Lemma}

\theoremstyle{definition}

\newtheorem{examp}[equation]{Example}

\newtheorem{quest}[equation]{Question}
\newtheorem{defn}[equation]{Definition}

\DeclareMathOperator{\gr}{grad}

\DeclareMathOperator{\im}{im}

\newcommand\Z{\mathbb{Z}}

\newcommand\Q{\mathbb{Q}}

\def\spinc {{\operatorname{spin^c}}}

\def\rk {{\operatorname{rank}}}

\def\fin\qedhere
\def\pr {{\text{pr}}}

\def\from {{\leftarrow}}

\def\s{\mathfrak s}

\def\gr{\mathrm{gr}}

\def\ff {{\mathbb{F}}}

\def\ker {{\operatorname{ker}}}

\def\fin\qedhere
\def\from {{\leftarrow}}

\newcommand{\ZZ}{\mathbb{Z}}

\def\H{\mathit{H}}

\def\CF {\mathit{CF}}
\def\HF {\mathit{HF}}

\newcommand\HFp {\HF^+}

\newcommand \CFm {\CF^-}
\newcommand \HFm {\HF^-}

\def\HFI {\mathit{HFI}}

\newcommand \HFIm {\HFI^-}

\newcommand \Hc {\H_{\mathrm{conn}}}

\def\inv{\iota}

\def\Inv{\mathfrak{I}}

\def\inv{\iota}
\def\ff {{\mathbb{F}}}
\def\im{\text{im }}
\def\ker{\text{ker }}
\def\gr{\text{gr}}
\def\width{\text{width}}

\title{Connected Heegaard Floer Homology of Sums of Seifert Fibrations}
\begin{document}
\author[Irving Dai]{Irving Dai}
\thanks{ID was partially supported by NSF grant DGE-1148900.}
\address{Department of Mathematics, Princeton University, Princeton, NJ 08540}
\email{idai@math.princeton.edu}

\begin{abstract}
We compute the connected Heegaard Floer homology (defined by Hendricks, Hom, and Lidman) for a large class of 3-manifolds, including all linear combinations of Seifert fibered homology spheres. We show that for such manifolds, the connected Floer homology completely determines the local equivalence class of the associated $\inv$-complex. Some identities relating the rank of the connected Floer homology to the Rokhlin invariant and the Neumann-Siebenmann invariant are also derived. Our computations are based on combinatorial models inspired by the work of N\'emethi on lattice homology.
\end{abstract}



\maketitle

\section{Introduction and Results}\label{sec:1}
Understanding the structure of the homology cobordism group $\Theta^3_{\ZZ}$ has been a motivating problem in the development of 3- and 4-manifold topology over the last several decades \cite{FSinstanton}, \cite{FurutaHom}, \cite{Froy}. Most recently, in \cite{Triangulations} Manolescu introduced a Pin(2)-equivariant version of Seiberg-Witten Floer homology and used this to refute the triangulation conjecture by ruling out 2-torsion in $\Theta^3_{\ZZ}$ with $\mu(Y) = 1$. Further variants and applications of Manolescu's construction have been studied extensively by several authors; see e.g.\ \cite{Lin}, \cite{Lin2}, \cite{Lin3}, \cite{Stoffregen}, \cite{Stoffregen2}, \cite{Dai}.

In \cite{HMinvolutive}, Hendricks and Manolescu defined an analogue of Pin(2)-equivariant Seiberg-Witten Floer homology in the Heegaard Floer setting, called involutive Heegaard Floer homology. This associates to a rational homology sphere $Y$ (with self-conjugate $\spinc$-structure $\s$) an algebraic object called an $\inv$-complex, whose mapping cone is a well-defined 3-manifold invariant $\HFIm(Y, \s)$.\footnote{Involutive Floer homology is defined for all three-manifolds, but in this paper we will only consider the case of rational homology spheres.} In \cite{HMZ}, Hendricks, Manolescu, and Zemke constructed a group $\Inv_{\Q}$, consisting of all possible $\inv$-complexes modulo an algebraic equivalence relation called local equivalence. This notion is modeled on the relation of homology cobordism, in the sense that if two (integer) homology spheres are homology cobordant, then their $\inv$-complexes are locally equivalent. Hendricks-Manolescu-Zemke obtain a homomorphism
\[
h: \Theta^3_{\ZZ} \rightarrow \Inv_{\Q}.
\]
One can thus obstruct the existence of a homology cobordism between two homology spheres by computing their local equivalence classes in $\Inv_{\Q}$ and showing that they are not equal. 

In \cite{DM}, Manolescu and the author exhibited a $\mathbb{Z}^\infty$-subgroup of $\Theta^3_{\ZZ}$ by establishing the linear independence of a certain infinite family of $\inv$-complexes $\{X_i\}_{i=1}^\infty$.\footnote{These differ  from the basis elements presented in \cite[Theorem 1.7]{DM} by an overall grading shift. See the discussion following \cite[Lemma 4.3]{DS}.} (This follows a similar proof in the Pin(2)-case due to Stoffregen \cite{Stoffregen2}; see also \cite{FSinstanton}, \cite{FurutaHom}.) Let $\mathfrak{X}$ be the subgroup of $\Inv_{\Q}$ spanned by the $X_i$. In \cite{DS}, Stoffregen and the author showed that (up to grading shift) $h(Y) \in \mathfrak{X}$ for a large class of 3-manifolds $Y$, including any sum of Seifert fibered homology spheres. For such $Y$, understanding $h(Y)$ thus reduces to studying linear combinations of the basis elements $X_i$. The author is currently unaware of any example with $h(Y) \notin \mathfrak{X}$; the existence of such a manifold would prove that Seifert fibered spaces do not generate the homology cobordism group.\footnote{There are many algebraic examples of $\inv$-complexes which are not locally equivalent to anything in $\mathfrak{X}$, but it is unclear whether these occur as the $\inv$-complexes of actual 3-manifolds.}

Following a construction of Stoffregen \cite{Stoffregen} in the Seiberg-Witten case, in \cite{HHL} Hendricks, Hom, and Lidman used involutive Heegaard Floer homology to define a new invariant of homology cobordism, called connected Heegaard Floer homology. This is an invariant of the local equivalence class of $(Y, \s)$ which associates to $(Y, \s)$ a particular (grading-shifted) summand $\HF_{\mathrm{conn}}(Y, \s)$ of the usual Heegaard Floer homology $\HFm(Y, \s)$. While \textit{a priori} the connected Floer homology of $Y$ is a strictly weaker invariant than its local equivalence class, it is considerably easier to state and understand, as the obstruction coming from $\HF_{\mathrm{conn}}$ takes the form of a homology group in the usual sense.

The goal of this paper is to use the results of \cite{DS} to compute the connected Floer homology of any 3-manifold with local equivalence class in $\mathfrak{X}$. Our main result (Theorem~\ref{thm:1.1}) establishes the possible forms that $\HF_{\mathrm{conn}}$ can take, and shows that (for such manifolds) $\HF_{\mathrm{conn}}$ is sufficient to recover the entire local equivalence class. We then re-phrase several computations from \cite{DS} to obtain identities relating the rank of $\HF_{\mathrm{conn}}$ to other homology cobordism invariants, such as the Rokhlin invariant and the Neumann-Siebenmann invariant \cite{Neu}, \cite{Sieb}. 

Finally, a brief word on the methods used in this paper. The bulk of our approach consists of developing combinatorial models for understanding the local equivalence classes of different $\inv$-complexes. We think that these models (which we call ``geometric complexes") are interesting in their own right, as they allow various algebraic constructions which are otherwise opaque to be translated into topological operations on finite-dimensional cell complexes. (See also \cite[Section 7]{DM}, \cite[Section 2.4]{DS}.) Our construction is heavily inspired by work of N\'emethi \cite{Nem} on lattice homology \cite{Plumbed}. Although here we only use geometric complexes as a convenient combinatorial tool, we hope that this may be indicative of deeper connections to Floer theory as a whole.

\subsection{Statement of Results}\label{sec:1.1} We first establish some notation. Let $Y$ be a rational homology sphere, and let $\s$ be a self-conjugate $\spinc$-structure on $Y$. We denote by $h$ the map sending $(Y, \s)$ to the local equivalence class of its (grading-shifted) $\inv$-complex:
\[
(Y, \s) \mapsto h(Y, s) = (\CFm(Y, \s), \inv)[-2].
\]
See Section~\ref{sec:2.1} for details. In this paper, we will be concerned with 3-manifolds for which $h(Y, \s)$ decomposes (up to grading shift) as a linear combination of the basis complexes $X_i$. Since each $X_i$ has $d$-invariant equal to zero, it is easily checked that in this situation the grading shift is given by $d(Y, \s)$; i.e.,
\[
h(Y, \s) = \left( \sum_i c_iX_i \right)[-d(Y, \s)].
\]
See Section~\ref{sec:2.1} for a definition of the $X_i$ and a discussion of which 3-manifolds are known to satisfy this decomposition. For now, the reader may take $Y$ to be a sum of Seifert fibered rational homology spheres. We begin by defining an algorithm that takes an element of $\mathbb{Z}^{\infty}$ (written as a linear combination of the $X_i$) and outputs a $\ZZ$-graded $\ff[U]$-module.

Let $\mathcal{T}_d(l)$ be the $U$-torsion tower $\ff[U]/U^l$, shifted in grading so that its uppermost element has grading $d$. Note that $\mathcal{T}_d(l)$ has $l$ elements; we refer to $l$ as the \textit{length} of $\mathcal{T}_d(l)$. We will additionally identify certain towers as \textit{downwards pointing} and other towers as \textit{upwards pointing}. In the first case, we define the \textit{head of} $\mathcal{T}_d(l)$ to be the element of maximal grading $d$, and the \textit{tail of} $\mathcal{T}_d(l)$ to be the element of minimal grading $d - 2(l - 1)$. In the second, we interchange these conditions, so that the head is the element of minimal grading and the tail is the element of maximal grading. Note that in both cases, the action of $U$ still decreases grading by two; these definitions are a notational aid and are not intrinsic to $\mathcal{T}_d(l)$.

Now let $X$ be a linear combination of the form
\[
X = \pm X_{i_1} \pm X_{i_2} \pm X_{i_3} \pm \cdots \pm X_{i_n},
\]
with $i_1 \geq i_2 \geq \cdots \geq i_n$. Without loss of generality, we may assume that $X$ is \textit{maximally simplified}, meaning that there are no cancelling pairs $X_i - X_i$. We use the indices appearing in $X$ to construct a $\mathbb{Z}$-graded $\ff[U]$-module, as follows. We begin by drawing a $U$-torsion tower of length $i_1$. If the sign of $X_{i_1}$ is positive, then we identify this tower as downwards pointing, and place it so that its head has grading zero. If the sign of $X_{i_1}$ is negative, then we identify the tower as upwards pointing, and place it so that its head has grading one.

We continue drawing $U$-torsion towers inductively, where at each step $k$ we draw a tower of length $i_k$. If the sign of $X_{i_k}$ is positive, then we identify this tower as downwards pointing; otherwise, we identify the tower as upwards pointing. We place the tower according to the following criteria:
\begin{enumerate}
\item If the sign of $X_{i_k}$ is opposite the sign of $X_{i_{k -1}}$, then the head of the $i_k$-tower should have the same grading as the tail of the $i_{k-1}$-tower.
\item If the signs of $X_{i_k}$ and $X_{i_{k-1}}$ are both positive, then the head of the $i_k$-tower should have grading one less than the tail of the $i_{k-1}$-tower.
\item If the signs of $X_{i_k}$ and $X_{i_{k-1}}$ are both negative, then the head of the $i_k$-tower should have grading one greater than the tail of the $i_{k-1}$-tower.
\end{enumerate}
We denote the $\ff[U]$-module constructed in this way by $\mathcal{H}(X)$. An example of this algorithm is presented in Figure~\ref{fig:1}.

\begin{figure}[h!]
\center
\includegraphics[scale=0.8]{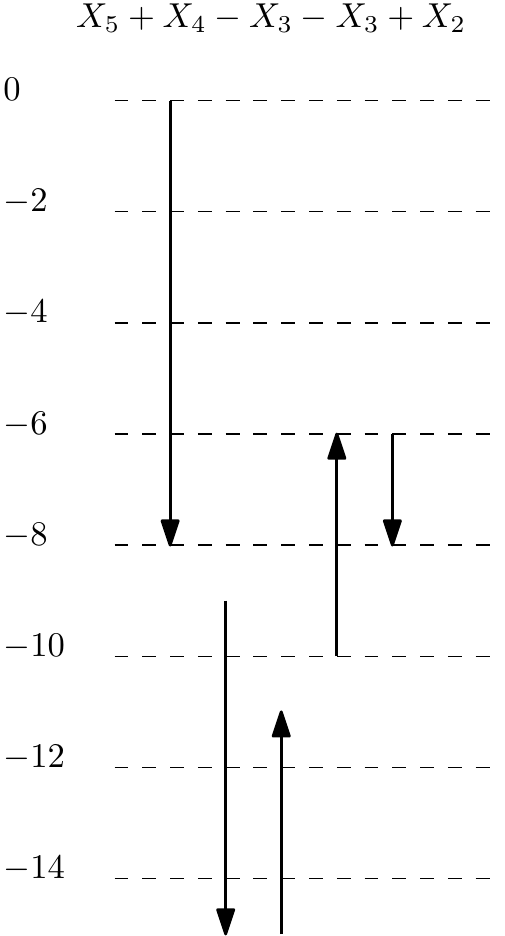}
\caption{Example of $\mathcal{H}(X)$. Here, arrows represent orientations of towers, where we think of an arrow as pointing from the head to the tail of a tower. Maslov gradings are indicated on the left of the diagram.}\label{fig:1}
\end{figure}

We are now ready to state our main theorem:
\begin{thm}\label{thm:1.1}
Let $Y$ be a rational homology sphere, and let $\s$ be a self-conjugate $\spinc$-structure on $Y$. Suppose that
\[
h(Y, \s) = (\pm X_{i_1} \pm X_{i_2} \pm X_{i_3} \pm \cdots \pm X_{i_n})[-d(Y, \s)],
\]
where we assume that the above linear combination is maximally simplified with $i_1 \geq i_2 \geq \cdots \geq i_n$. Then the connected Heegaard Floer homology of $(Y, \s)$ is given by
\[
\HF_{\mathrm{conn}}(Y, \s) = \mathcal{H}(X)[-d(Y, \s) + 1].
\]
Here, $X = \pm X_{i_1} \pm \cdots \pm X_{i_n}$ is the linear combination obtained in the decomposition of $h(Y, \s)$, and $\mathcal{H}(X)$ is the $\ff[U]$-module constructed using the algorithm defined previously.
\end{thm}

We will prove Theorem~\ref{thm:1.1} by establishing an explicit model for the connected complex (see Definition~\ref{defn:2.7}) of $(Y, \s)$. Conversely, suppose that we have some $h(Y, \s) \in \mathfrak{X}$ for which $\HF_{\mathrm{conn}}(Y, \s)$ is known. Then it is not hard to check that one can recover the decomposition of $h(Y, \s)$ into the $X_i$. Hence:

\begin{thm} \label{thm:1.2}
Suppose that (up to grading shift) $h(Y, \s) \in \mathfrak{X}$. Then the local equivalence class of $h(Y, \s)$ in $\Inv_{\Q}$ is completely determined by the isomorphism class of $\HF_{\mathrm{conn}}(Y, \s)$, together with $d(Y, \s)$. 
\end{thm}

If $Y_1$ and $Y_2$ are two manifolds as above whose connected Floer homologies and $d$-invariants are known, then Theorems~\ref{thm:1.1} and \ref{thm:1.2} allow us to immediately compute the connected Floer homology of $Y_1 \# Y_2$. To do this, we first apply Theorem~\ref{thm:1.2} to write $h(Y_1)$ and $h(Y_2)$ as a linear combination of the $X_i$; then we sum these together, remove all cancelling pairs, and apply Theorem~\ref{thm:1.1} to the result.

We also derive some identities involving the rank of $\HF_{\mathrm{conn}}(Y)$. For these applications, we specialize to the class of almost-rational (AR) plumbed 3-manifolds. This family (which includes all Seifert fibered rational homology spheres) is defined by placing certain combinatorial restrictions on the usual three-dimensional plumbing construction. See \cite{Nem} (also \cite[Definition 2.7]{DM}) for details.  

\begin{cor} \label{cor:1.3}
Let $Y$ be a linear combination of \text{AR} plumbed 3-manifolds, and let $\s$ be a self-conjugate $\spinc$-structure on $Y$. Then 
\[
\bar{\mu}(Y, \s) + \dfrac{d(Y, \s)}{2} = \mathrm{signed}\ \rk \ \HF_{\mathrm{conn}}(Y, \s).
\]
Here, $\bar{\mu}(Y, \s)$ is the Neumann-Siebenmann invariant of $(Y, \s)$. The signed rank of $\HF_{\mathrm{conn}}$ is computed by assigning each element in a downwards pointing tower a positive multiplicity, and each element in an upwards pointing tower a negative multiplicity.
\end{cor}

Note that the orientations of the towers in $\HF_{\mathrm{conn}}(Y, \s)$ are not intrinsically defined, but must instead be deduced from the overall structure of $\HF_{\mathrm{conn}}(Y, \s)$. (See the proof of Theorem~\ref{thm:1.2}.) This can only be done when $\HF_{\mathrm{conn}}(Y, \s)$ takes the form prescribed by Theorem~\ref{thm:1.1}, which is consistent with the fact that $\bar{\mu}$ is not defined for all 3-manifolds.
\begin{cor} \label{cor:1.4}
Let $Y$ be a linear combination of AR plumbed homology spheres. Then 
\[
\mu(Y) + \dfrac{d(Y)}{2} = \rk \ \HF_{\mathrm{conn}}(Y) \ \ (\mathrm{mod}\ 2),
\]
where $\mu(Y)$ is the Rokhlin invariant of $Y$. 
\end{cor}

Finally, we end with the following question:

\begin{quest}\label{q:1.5}
Do there exist (integer) homology spheres $Y$ for which $\HF_{\mathrm{conn}}(Y)$ is \textit{not} isomorphic to an $\ff[U]$-module of the form described in Theorem~\ref{thm:1.1}?
\end{quest}

An affirmative answer would produce an element of $\Theta^3_{\ZZ}$ not homology cobordant to any sum of Seifert fibered spaces. There are certainly algebraic examples of $\inv$-complexes whose connected homology does not take the form of Theorem~\ref{thm:1.1} (see for instance Example~\ref{ex:6.2}), but it is unclear whether these occur as the $\inv$-complexes of actual 3-manifolds.

\medskip
\noindent {\bf Organization of the paper.} In Section~\ref{sec:2}, we review the algebraic formalism underlying involutive Heegaard Floer homology and connected Floer homology. In Section~\ref{sec:3}, we discuss geometric complexes and outline several related constructions. Sections~\ref{sec:4} and \ref{sec:5} are devoted to proving Theorem~\ref{thm:1.1}. Finally, in Section~\ref{sec:6}, we prove Theorem~\ref{thm:1.2}, as well as Corollaries~\ref{cor:1.3} and \ref{cor:1.4}.

\medskip
\noindent {\bf Acknowledgements.} The author would like to thank his advisor, Zolt\'an Szab\'o, for his continued support and guidance. He would also like to thank Kristen Hendricks, Jen Hom, Tye Lidman, Ciprian Manolescu, and Matt Stoffregen for helpful conversations. In addition, the author would like to thank the referee for several helpful comments.


\section{Background}\label{sec:2}

\subsection{Involutive Floer Homology}\label{sec:2.1} In this section, we review the involutive Heegaard Floer package defined by Hendricks and Manolescu \cite{HMinvolutive}. For the present application, we will only need to understand the algebraic formalism of the output; we refer the interested reader to \cite{HMinvolutive} instead for the topological details of the construction. 

\begin{defn}\cite[Definition 8.1]{HMZ}
An {\em $\inv$-complex} is a pair $(C, \inv)$, consisting of
\begin{itemize}
\item a $\Q$-graded, finitely generated, free chain complex $C$ over the ring $\ff[U]$, where $\operatorname{deg}(U)=-2$. Moreover, we ask that there is some $\tau \in \Q$ such that the complex $C$ is supported in degrees differing from $\tau$ by integers. We also require that there is a relatively graded isomorphism
\[
U^{-1}H_*(C) \cong \ff[U, U^{-1}],
\]
and that $U^{-1}H_*(C)$ is supported in degrees differing from $\tau$ by even integers;
\item a grading-preserving chain homomorphism $\inv: C \to C$, such that $\inv^2$ is chain homotopic to the identity.
\end{itemize}
Two $\inv$-complexes are said to be \textit{homotopy equivalent} if there are (grading-preserving) homotopy equivalences between them which are $\inv$-equivariant up to homotopy. 
\end{defn}

Given a rational homology sphere $Y$ with self-conjugate $\spinc$-structure $\s$, Hendricks and Manolescu define a homotopy involution $\inv$ on the Heegaard Floer complex $\CFm(Y, \s)$ by using the involution on the Heegaard diagram interchanging the $\alpha$- and $\beta$-curves. They then show that the map sending
\[
(Y, \s) \mapsto (\CFm(Y, \s), \inv)
\] 
is well-defined up to homotopy equivalence of $\inv$-complexes. The involutive Heegaard Floer homology of $(Y, \s)$ is constructed by taking the mapping cone of $\CFm(Y, \s)$ with respect to the map $1 + \inv$. In this paper, we will always work with the $\inv$-complexes themselves, rather than their involutive homology.

\begin{defn}\label{defn:2.2}
For any positive integer $i$, let $X_i$ be the chain complex generated over $\ff[U]$ by $\alpha$, $\iota \alpha$, and $\beta$, with
\[
\partial \beta = U^i(\alpha + \inv \alpha)
\]
and $\partial(\alpha) = \partial(\inv\alpha) = 0$. The Maslov gradings are given by $M(\alpha) = M(\inv \alpha) = 0$ and $M(\beta) = -2i + 1$. The action of $\inv$ interchanges $\alpha$ and $\inv \alpha$ and sends $\beta$ to itself. Note that here $\inv$ is an actual involution (rather than a homotopy involution); this will be true for all complexes considered in this paper.
\end{defn}

In order to study homology cobordism, we will need the following weaker notion of equivalence between $\inv$-complexes:

\begin{defn}
Two $\inv$-complexes $(C, \inv)$ and $(C', \inv')$ are called {\em locally equivalent} if there exist (grading-preserving) chain maps
\[
f : C \to C', \ \ g : C' \to C
\]
between them such that 
\[
f \circ \inv \simeq \inv' \circ f,  \ \ \ g \circ \inv' \simeq \inv \circ g,
\]
and $f$ and $g$ induce isomorphisms on homology after inverting the action of $U$. Note that this is strictly weaker than the notion of homotopy equivalence between $\inv$-complexes.
\end{defn}

We call a map $f$ as above a \textit{local map} from $(C, \inv)$ to $(C', \inv')$, and similarly we refer to $g$ as a local map in the other direction. In \cite{HMZ}, it is shown that if $Y_1$ and $Y_2$ are homology cobordant, then their respective $\inv$-complexes are locally equivalent.

In \cite[Section 8]{HMZ}, it is shown that the set of $\inv$-complexes up to local equivalence forms a group, with the group operation being given by tensor product. We call this group the \textit{involutive Floer group} and denote it by $\Inv_{\Q}$:

\begin{defn}\cite[Proposition 8.8]{HMZ}
Let $\Inv_{\Q}$ be the set of $\inv$-complexes up to local equivalence. This has a multiplication given by tensor product, which sends (the local equivalence classes of) two $\inv$-complexes $(C_1, \inv_1)$ and $(C_2, \inv_2)$ to (the local equivalence class of) their tensor product complex $(C_1 \otimes C_2, \inv_1 \otimes \inv_2)$. 
\end{defn}

The identity element of $\Inv_{\Q}$ is given by the trivial complex $\mathcal{I}_0$ consisting of a single $\ff[U]$-tower starting in grading zero, together with the identity map on this complex. Inverses in $\Inv_{\Q}$ are given by dualizing. There is an obvious subgroup of $\Inv_{\Q}$ generated by the set of $\inv$-complexes which are $\Z$-graded; we denote this by $\Inv$. Clearly, $\Inv_{\Q}$ consists of an infinite number of copies of $\Inv$, one for each $[\tau] \in \Q/2\Z$.

Let $h$ be the map sending a pair $(Y, \s)$ to the local equivalence class of its (grading-shifted) $\inv$-complex:
\[
h(Y, \s) = (\CFm(Y, \s), \iota)[-2].
\]
The main result of \cite{DS} is then:\footnote{This is not exactly the statement of \cite[Theorem 1.2]{DS}; the basis elements presented in Theorem~\ref{thm:2.5} differ from the ones in \cite{DM} and \cite{DS} by an overall grading shift. }
\begin{thm}\cite[Theorem 1.2]{DS}\label{thm:2.5}
Let $Y$ be a linear combination of almost-rational plumbed 3-manifolds, and let $\s$ be a self-conjugate $\spinc$-structure on $Y$. Then $h(Y, \s)$ decomposes uniquely as
\[
h(Y, \s) = \left( \sum_i c_iX_i \right)[-d(Y, \s)].
\]
\end{thm}
\noindent
Here, we define the family of almost-rational (AR) plumbed 3-manifolds by placing certain combinatorial restrictions on the usual three-dimensional plumbing construction. See \cite{Nem} (also \cite[Definition 2.7]{DM}) for a precise discussion. Note that the class of AR plumbed 3-manifolds includes all Seifert fibered rational homology spheres.

There are other classes of manifolds for which the conclusion of Theorem~\ref{thm:2.5} holds, most notably linear combinations of manifolds whose local equivalence classes are already given by the $X_i$. This is true (in particular) for negative surgeries on $L$-space knots; see the proof of \cite[Theorem 1.8]{HHL}.

\subsection{Connected Floer Homology}\label{sec:2.2} We now turn to the definition of connected Floer homology \cite{HHL}. We begin with:
\begin{defn}\cite[Definition 3.1]{HHL}
Let $(C, \inv)$ be an $\inv$-complex. A \textit{self-local equivalence} is a local map from $(C, \inv)$ to itself; that is, a (grading-preserving) chain map
\[
f: C \to C,
\]
such that
\[
f \circ \inv \simeq \inv \circ f,
\]
and $f$ induces an isomorphism on homology after inverting the action of $U$.
\end{defn}

Hendricks, Hom, and Lidman define a preorder on the set of self-local equivalences of $(C, \inv)$ by declaring $f \lesssim g$ whenever $\ker f \subseteq \ker g$. A self-local equivalence $f$ is then said to be \textit{maximal} if $g \gtrsim f$ implies $g \lesssim f$ for any self-local equivalence $g$.

Maximal self-local equivalences should heuristically be thought of as producing a local equivalence between $C$ and some very small subcomplex of $C$ given by $\im f$. (This subcomplex is small because $\ker f$ is large.) Note, however, that since $\inv$ is a homotopy involution and $f$ commutes with $\inv$ only up to homotopy, the action of $\inv$ need not preserve $\im f$, and similarly for the relevant homotopy maps. However, according to \cite[Lemma 3.7]{HHL}, we can modify $\inv$ to produce an actual homotopy involution $\inv_f$ on $\im f$ such that $(C, \inv)$ and $(\im f, \inv_f)$ are locally equivalent. 

Hendricks-Hom-Lidman further show that a maximal self-local equivalence always exists \cite[Lemma 3.3]{HHL}, and that any two maximal self-local equivalences give homotopy equivalent $\inv$-complexes $(\im f, \inv_f)$ and $(\im g, \inv_g)$ \cite[Lemma 3.8]{HHL}. We can thus make the unambiguous definition:

\begin{defn}\cite[Definition 3.9]{HHL}\label{defn:2.7}
Let $(C, \inv)$ be an $\inv$-complex. The \textit{connected complex of} $(C, \inv)$, which we denote by $(C_{\mathrm{conn}}, \inv_{\mathrm{conn}})$, is defined to be $(\im f, \inv_f)$, where $f$ is any maximal self-local equivalence. This is well-defined up to homotopy equivalence of $\inv$-complexes.
\end{defn}

Note that since $(C, \inv)$ and $(C_{\mathrm{conn}}, \inv_{\mathrm{conn}})$ are locally equivalent, the connected complex is in fact an invariant of the local equivalence class of $(C, \inv)$ \cite[Proposition 3.10]{HHL}. Indeed, the connected complex should be thought of as the simplest possible local representative of $(C, \inv)$. The connected homology is then defined to be (more or less) the usual homology of this complex:

\begin{defn}\cite[Definition 3.13]{HHL}
Let $(C, \inv)$ be an $\inv$-complex. The \textit{connected homology of} $(C, \inv)$ is defined to be the $U$-torsion submodule of $H_*(C_{\mathrm{conn}})$, shifted up in grading by one. Here, $H_*(C_{\mathrm{conn}})$ is the usual homology of $C_{\mathrm{conn}}$ as a $\ff[U]$-complex.
\end{defn}

The \textit{connected Heegaard Floer homology of} $(Y, \s)$ is then of course defined to be the connected homology of $(\CFm(Y, \s), \inv)$. The grading shift by one is enacted so that $\HF_\mathrm{conn}$ is a summand of $\HF_\mathrm{red}$, viewed as a quotient of $\HFp$. 

\begin{examp}\label{examp:2.9}
The homology of $X_i$ is given by the direct sum of a single non-torsion tower (starting in grading zero) and a $U$-torsion tower $\mathcal{T}_0(i)$ of length $i$. It is not difficult to show that the connected complex of $X_i$ is all of $X_i$, and hence that the connected homology of $X_i$ is precisely $\mathcal{T}_0(i)$ (shifted up in grading by one). See \cite[Lemma 7.2]{HHL}. 


\end{examp}

In \cite[Proposition 4.1]{HHL}, it is shown that the connected Floer homology behaves well under dualization. Note, however, that the connected homology does \textit{not} behave well under tensor product. Roughly speaking, this is due to the fact that even if $(C_1, \inv_1)$ and $(C_2, \inv_2)$ are the simplest representatives of their respective local equivalence classes, the tensor product $(C_1 \otimes C_2, \inv_1 \otimes \inv_2)$ need not be.

In this paper, we will not actually use any properties of the connected Floer homology beyond the fact that it is a (grading-shifted) summand of the usual Floer homology and that it dualizes well. Instead, in light of Theorem~\ref{thm:2.5}, we will focus on producing ``simple" representatives of linear combinations of the $X_i$. While each of the $X_i$ is individually very straightforward, we will see that it is actually quite subtle to construct good local representatives of their tensor products.


\section{Geometric Operations}\label{sec:3}

In this section, we review the definition of a geometric complex and outline several related constructions which we will use to describe the $\inv$-complexes of sums of Seifert fibered spaces. See \cite[Section 7]{DM} and \cite[Section 2.4]{DS} for background on geometric complexes and associated applications.

\subsection{Geometric Complexes} \label{sec:3.1} Let $C_{cell}$ be a finite cell complex. Suppose that we are given a function $\gr$ from the cells of $C_{cell}$ to a coset of $2\ZZ$ in $\mathbb{Q}$, with the property that whenever a cell $e_{i-1}$ appears in the cellular boundary of $e_i$, we have $\gr(e_{i-1}) \geq \gr(e_i)$. Let $C = C_{cell} \otimes \ff[U]$. We turn $C$ into a (relatively) $\mathbb{Z}$-graded chain complex over $\ff[U]$ as follows. For any cell $e_i$ of dimension $i$, we define the Maslov grading $M(e_i)$ of $e_i$ to be
\[
M(e_i) = \gr(e_i) + i.
\] 
We extend this to all of $C$ by declaring $U$ to be of degree $-2$. The differential on $C$ is defined to be
\[
\partial e_i = \sum_{e_{i-1} \in \text{ cellular bdry of } e_i} U^{(\gr(e_{i-1})-\gr(e_i))/2}e_{i-1}.
\]
Here, we have simply taken the usual cellular differential and multiplied by the appropriate powers of $U$ everywhere so as to be of degree $-1$ with respect to the Maslov grading. Note that each exponent of $U$ is non-negative due to the condition on $\gr$. We call a chain complex which is defined in this manner a \textit{geometric complex}, and we refer to the underlying cell complex $C_{cell}$ as the \textit{skeleton} of $C$. Note that $C$ comes equipped with an additional grading coming from the usual $\ZZ$-grading on $C_{cell}$, which may be extended to all of $C$ by declaring $U$ to have degree zero (with respect to this grading). We refer to this as the \textit{dimensional grading}.

More conceptually, all we are doing is presenting a convenient bookkeeping device for thinking about $C$. We view the data of $C$ as separated into two distinct parts - first, the skeleton $C_{cell}$, which is a chain complex in the usual sense over $\ff$; and second, the data consisting of the Maslov gradings of the generators of $C$, which is encoded in the function $\gr$ together with the dimensional grading. We think of the skeleton as giving the ``shape" of differential on $C$, with the correct powers of $U$ in the differential being forced by the knowledge of the Maslov gradings.

In this paper, all of our geometric complexes will come equipped with a cellular chain map $J$ on $C_{cell}$ with $J^2 = 1$. Geometrically, this corresponds to reflection through some point in $C_{cell}$. If the grading function $\gr$ is invariant with respect to $J$, then by imposing $U$-equivariance we can extend $J$ to a grading-preserving chain involution on all of $C$ (which we also denote by $J$). Moreover, in all of the cases we will be discussing, the action of $J$ will have exactly one fixed cell $\eta$ with $J\eta = \eta$; all other cells occur in $J$-symmetric pairs. We refer to a geometric complex with this property as a \textit{split complex}.

\begin{examp}\label{examp:3.1}
The prototypical example of a (split) geometric complex is the complex $X_i$ of Definition~\ref{defn:2.2}. The skeleton of $X_i$ is given by a (rather trivial) one-dimensional cell complex, consisting of two 0-cells corresponding to $\alpha$ and $J\alpha$, and a single 1-cell corresponding to $\beta$. The action of $J$ interchanges $\alpha$ and $J\alpha$ and sends $\beta$ to itself. This complex, together with its associated grading function $\gr$, is displayed on the right in Figure~\ref{fig:3}. See \cite[Section 2.4]{DS} for further discussion.

\begin{figure}[h!]
\center
\includegraphics[scale=1]{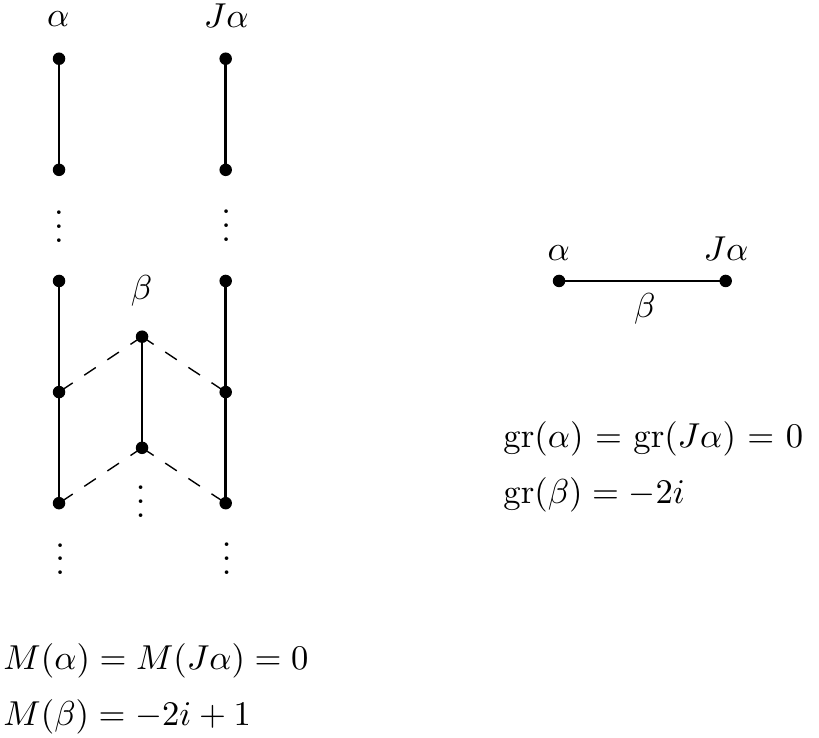}
\caption{The complex $X_i$ (left) and its associated skeleton (right). On the left, solid lines represent multiplication by $U$ and dashed lines represent the action of $\partial$.}\label{fig:3}
\end{figure}
\end{examp}
If $C_{cell}$ has trivial homology, then it is easy to check that the pair $(C = C_{cell} \otimes \ff[U], J)$ forms an $\inv$-complex. When we say that an $\inv$-complex is locally equivalent to some geometric complex with involution $J$, we of course mean that our $\inv$-complex is locally equivalent to $(C, J)$ in the usual sense. See \cite[Lemma 3.3]{DS}.

If $C_1$ and $C_2$ are two geometric complexes, then their tensor product is also a geometric complex, constructed as follows. The skeleton of $C_1 \otimes C_2$ is given by the usual cellular product of the skeleta of $C_1$ and $C_2$, and the grading function $\gr$ is defined by setting $\gr(x \times y) = \gr(x) + \gr(y)$ for any pair of cells $x$ from $C_1$ and $y$ from $C_2$. If $C_1$ and $C_2$ are split, then $C_1 \otimes C_2$ is also split, with the involution on $C_1 \otimes C_2$ acting coordinatewise. See \cite[Section 2.4]{DS}.

We now give two important families of geometric complexes:
\begin{examp}\label{examp:3.2}
Consider an alternating sum of $X_i$; i.e.,
\[
X_{i_1} - X_{i_2} + X_{i_3} - \cdots \pm X_{i_n}
\]
with $i_1 > i_2 > \cdots > i_n$. By results of \cite{DS}, this is locally equivalent to a one-dimensional geometric complex. (According to \cite[Theorem 4.2]{DS}, the above sum is locally equivalent to a monotone root, which in turn can be described by a one-dimensional geometric complex as constructed in \cite[Section 7]{DM}.) There are two possible forms for this complex, both of which are displayed in Figure~\ref{fig:4}. If $n$ is even (so that the sign in front of $X_{i_n}$ is a minus), then the skeleton of the relevant complex has a $J$-invariant 0-cell; otherwise, the skeleton has a $J$-invariant 1-cell. We have labeled the 0-cells of our complex as $v_k$ and $Jv_k$ (with $k$ even) and the 1-cells of our complex as $w_k$ and $Jw_k$ (with $k$ odd). The grading function $\gr$ is given by
\begin{align*}
&\gr(v_k) = \gr(Jv_k) = 2(-i_1 + i_2 - \cdots + i_k), \text{ and} \\
&\gr(w_k) = \gr(Jw_k) = 2(-i_1 + i_2 - \cdots - i_k),
\end{align*}
with $\gr(v_0) = \gr(Jv_0) = 0$.

\begin{figure}[h!]
\center
\includegraphics[scale=1]{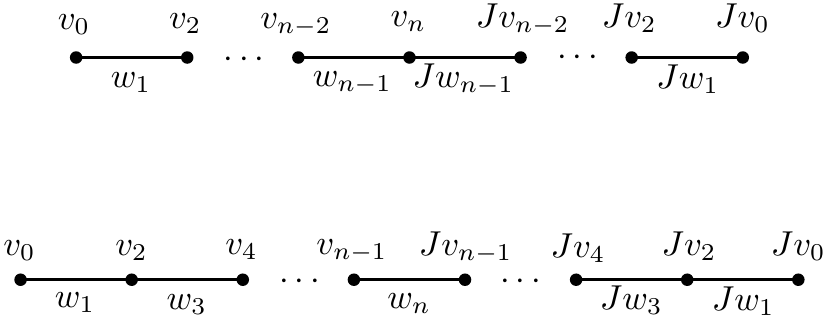}
\caption{Geometric complexes of alternating sums with $n$ even (top) and $n$ odd (bottom).}\label{fig:4}
\end{figure}

It is straightforward to compute the (usual) homology of this complex and show that its torsion part is given by the procedure described in Theorem~\ref{thm:1.1}. (To see this, it is easier to think of the complex as a monotone root, as in \cite[Section 6]{DM}.) The result is schematically depicted in Figure~\ref{fig:5}. By \cite[Proposition 7.5]{HHL}, this is indeed equal (after grading shift) to the connected homology of $X_{i_1} - X_{i_2} + \cdots \pm X_{i_n}$.

\begin{figure}[h!]
\center
\includegraphics[scale=0.8]{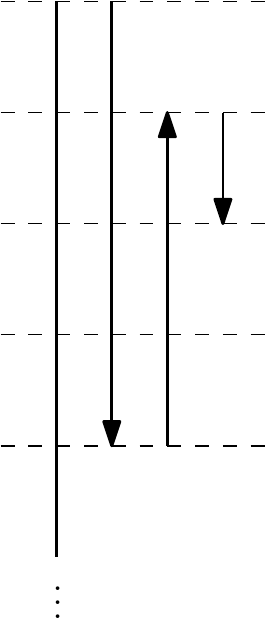}
\caption{Homology of an alternating sum ($n$ odd). Dotted lines are separated by Maslov grading two, with the uppermost line having Maslov grading zero. This example is $X_5 - X_4 + X_2$.}\label{fig:5}
\end{figure}

\end{examp}

\begin{examp}\label{examp:3.3}
Consider a sum of $X_i$ with the same orientation; i.e.,
\[
X_{i_1} + X_{i_2} + X_{i_3} + \cdots + X_{i_n}
\]
with $i_1 \geq i_2 \geq \cdots \geq i_n$. By \cite[Lemma 5.2]{DS}, this is locally equivalent to a geometric complex whose skeleton is a $J$-symmetric decomposition of the $n$-ball. This is displayed in Figure~\ref{fig:6}. For each $0 \leq k < n$, there are two cells of dimension $k$, which we have labelled $e_k$ and $Je_k$, while there is a single $J$-symmetric cell $e_n$ of dimension $n$. The grading function $\gr$ is given by
\[
\gr(e_k) = \gr(Je_k) = 2(-i_1 - i_2 - \cdots - i_k),
\]
with $\gr(e_0) = \gr(Je_0) = 0$.

\begin{figure}[h!]
\center
\includegraphics[scale=1]{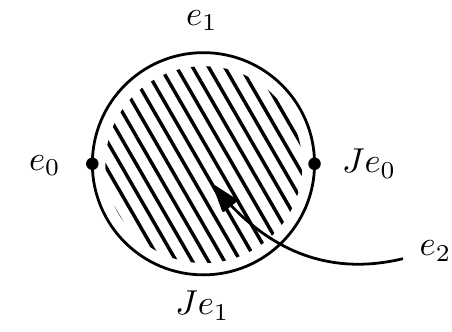}
\caption{Geometric complex of same-sign sums.}\label{fig:6}
\end{figure}

Again, it is straightforward to compute the (usual) homology of this complex and show that its torsion part is given the procedure described in Theorem~\ref{thm:1.1}. (Generators of the torsion towers are given by the obvious cycles $e_k + Je_k$.) The result is schematically depicted in Figure~\ref{fig:7}. By \cite[Proposition 7.1]{HHL}, this is indeed equal (after grading shift) to the connected homology of $X_{i_1} + X_{i_2} + \cdots + X_{i_n}$.

\begin{figure}[h!]
\center
\includegraphics[scale=0.6]{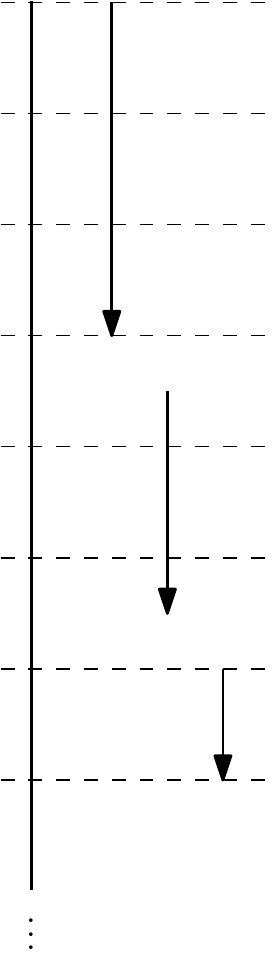}
\caption{Homology of same-sign sums. Dotted lines are separated by Maslov grading two, with the uppermost line having Maslov grading zero. This example is $X_4 + X_3 + X_2$.}\label{fig:7}
\end{figure}
\end{examp}

Examples~\ref{examp:3.2} and \ref{examp:3.3} give geometric representatives for the local equivalence classes of alternating sums and same-sign sums of $X_i$, respectively. In both cases, the given representatives have the property that the connected homology is given by (the torsion part of) the usual homology. This is not very surprising, as one might guess that the complexes in question are indeed the simplest possible representatives of their local equivalence classes. (Taking the tensor product of $n$ of the $X_i$ without any simplification yields a complex with $3^n$ generators; our model complex has $2n + 1$ generators.) The goal for the remainder of this paper will be to produce similar representatives for arbitrary linear combinations of the $X_i$.

Before we continue, we will need to make one more definition involving geometric complexes. Let $C$ be a geometric complex with skeleton $C_{cell}$ and grading function $\gr$. For each cell $e_i$ of $C_{cell}$, we define
\[
\width(e_i) = \min_{e_{i-1} \in \text{ cellular bdry of } e_i} \{ \gr(e_{i-1})-\gr(e_i) \},
\]
with the understanding that if $\partial e_i = 0$ then $\width(e_i) = \infty$. We define the \textit{width} of the complex $C$ to be the minimum width over all its cells; i.e.,
\[
\width(C) = \min\{ \width(e_i) \}.
\]
Note that this means any term appearing in the boundary map $\partial$ on $C$ must appear with a $U$-power of at least $\width(C)/2$.

\subsection{Doubling Operations}\label{sec:3.2}
In this section, we sketch some constructions which will allow us to inductively produce local representatives for linear combinations of the $X_i$. Our exposition here is meant to convey intuition and form a geometric overview; see Section~\ref{sec:4.1} for the precise details.

Consider the cell complexes displayed in Figure~\ref{fig:8}. On the left, we have displayed the skeleta of several split complexes, while on the right, we have displayed a corresponding series of cell complexes obtained in each case by ``doubling" the single $J$-invariant cell on the left. More precisely, to pass from the left- to the right-hand side, we replace the single $J$-invariant cell $\eta$ on the left with a new pair of symmetric cells (denoted by $\omega$ and $J\omega$), each of which have the same boundary (if any) as $\eta$. We then fill in the resulting ``gap" with a single $J$-invariant cell $\theta$ of one higher dimension, whose boundary is the sum of our two new cells. 

\begin{figure}[h!]
\center
\includegraphics[scale=0.9]{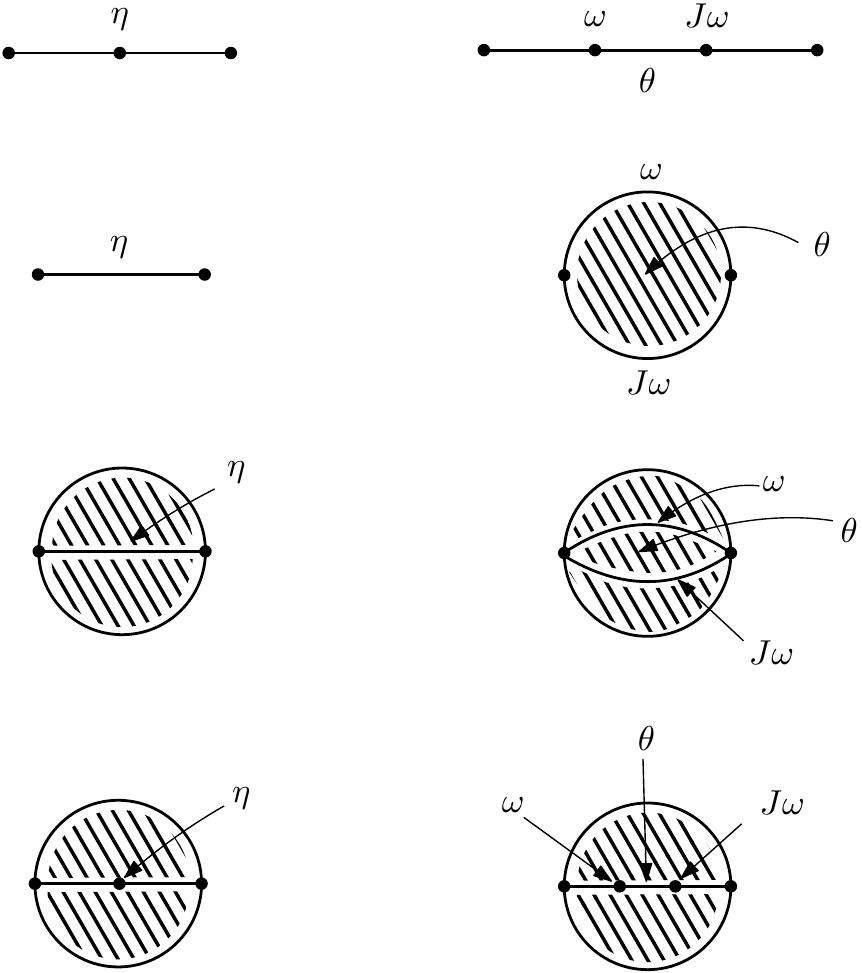}
\caption{Doublings of split complexes.}\label{fig:8}
\end{figure}

Now let $X$ be any split complex, and let $\Delta$ be any non-negative integer with $2\Delta \leq \width(X)$. We define a new split complex $X^d(\Delta)$, called the \textit{double} of $X$ (with parameter $\Delta$), as follows. The skeleton of $X^d(\Delta)$ is given by applying the above doubling operation to the skeleton of $X$. The function $\gr$ on $X^d(\Delta)$ is defined as follows:
\begin{enumerate}
\item On every cell except $\omega$, $J\omega$, and $\theta$, the function $\gr$ is equal to its value on the corresponding cell from $X$;
\item $\gr(\omega) = \gr(J\omega) = \gr(\eta)$; and,
\item $\gr(\theta) = \gr(\eta) - 2\Delta$.
\end{enumerate}
\noindent
In Section~\ref{sec:4.1} we will show that $X^d(\Delta)$ is a well-defined split geometric complex (i.e., that the function $\gr$ satisfies the usual monotonicity condition with respect to the cellular differential) and explain the condition $2\Delta \leq \width(X)$. For now, the reader should take a moment to verify that for the representative complexes of Example~\ref{examp:3.2}, the alternating sum $X_{i_1} - X_{i_2} + \cdots + X_{i_n}$ (for $n$ odd) is the double (with parameter $i_n$) of $X_{i_1} - X_{i_2} + \cdots - X_{i_{n-1}}$. Similarly, for the complexes of Example~\ref{examp:3.3}, we have that $X_{i_1} + X_{i_2} + \cdots + X_{i_n}$ is the double (with parameter $i_n$) of $X_{i_1} + X_{i_2} + \cdots + X_{i_{n-1}}$. Compare with the first two examples in Figure~\ref{fig:8}.

The main technical result of this paper will be to show that the doubled complex $X^d(\Delta)$ is locally equivalent to $X \otimes X_{\Delta}$ in general. This will allow us to construct local representatives of linear combinations of the $X_i$ via induction. Note that passing from $X \otimes X_{\Delta}$ to $X^d(\Delta)$ represents significant reduction in complexity, since if $X$ has $g$ generators, then $X \otimes X_{\Delta}$ has $3g$ generators, while $X^d(\Delta)$ has $g + 2$ generators.

We should stress here that for very simple complexes, it is straightforward to describe the cellular boundary map of the doubled complex in terms of the original. For example, the double of a spherical complex is a spherical complex of one higher dimension. However, sometimes the relation between the two cellular boundary maps is slightly more subtle. For instance, in the fourth example of Figure~\ref{fig:8}, we must modify the boundary map on the 2-cells incident to our original $J$-invariant 0-cell to include the new $J$-invariant 1-cell $\theta$.

The reader may thus (rightly) object that we have not really defined how to construct the skeleton of $X^d(\Delta)$. While in the examples above it is more or less obvious how to ``double" the $J$-invariant cell of $X$, one might wonder whether this is always well-defined, or indeed whether there may be multiple ways to carry out this operation. For example, Figure~\ref{fig:9} shows a split complex with two putative doublings. In Section~\ref{sec:4.1} we will provide a more precise definition of the doubling operation by explicitly defining the doubled complex and its boundary map; this will turn out to depend on an additional set of choices. However, we will in fact show that \textit{any} choice of doubling is locally equivalent to $X \otimes X_{\Delta}$, so for our purposes this ambiguity is unimportant.

\begin{figure}[h!]
\center
\includegraphics[scale=0.6]{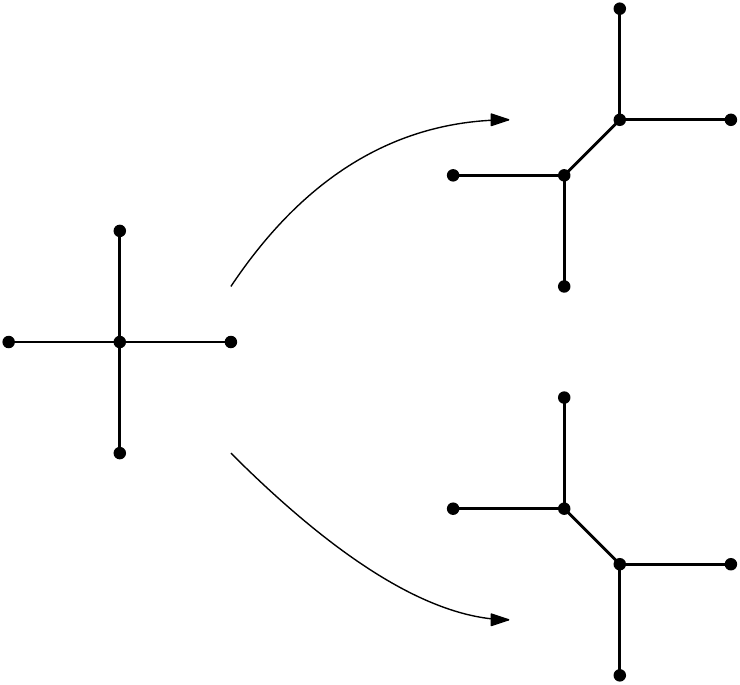}
\caption{A split complex which appears to have two possible doublings.}\label{fig:9}
\end{figure}

\subsection{Halving Operations}\label{sec:3.3}
Let us briefly review the procedure for dualizing $\inv$-complexes given in \cite[Section 8.3]{HMZ}. If $\{x_i\}$ is a homogenous $\ff[U]$-basis for $X$, then the dual complex $X^*$ has a corresponding $\ff[U]$-basis given by $\{x_i^*\}$, where the Maslov grading of $x_i^*$ is $M(x_i^*) = - M(x_i)$. The differential on $X^*$ is defined so that $U^kx_b^*$ appears in $\partial^* x_a^*$ if and only if $U^k x_a$ appears in $\partial x_b$. Similarly, an element $U^k x_b^*$ appears in $\inv^* x_a^*$ if and only if $U^k x_a$ appears in $\inv x_b$. Note that the action of $U$ on $X^*$ still has degree $-2$, and the differential still drops the Maslov grading by one.

If $X$ is a geometric complex, then $X^*$ is not quite a geometric complex in the sense defined previously. For example, the dual of $X_i$ has three generators $\alpha^*$, $J\alpha^*$, and $\beta^*$, with the differential
\[
\partial^* (\alpha^*) = \partial^* (J\alpha^*) = U^i \beta^*
\]
and $\partial^* \beta^* = 0$. Hence the putative ``skeleton" of $X_i^*$ has two 1-cells (corresponding to $\alpha^*$ and $J\alpha^*$), each having boundary consisting of only one 0-cell (corresponding to $\beta^*$). While this evidently does not correspond to a cell complex in the usual sense, it is not difficult to see that we may view this as a \textit{relative} cell complex, as in Figure~\ref{fig:10}. Here, the homology is no longer isomorphic to the usual singular homology of our cell complex, but rather the relative homology of our cell complex with respect to some subcomplex.

We thus relax the definition of a geometric complex slightly by not requiring the skeleton to be an actual cell complex, but simply a (finite) $\mathbb{Z}$-graded chain complex over $\ff$. (We will still refer to generators of the skeleton as cells, and the $\mathbb{Z}$-grading as the dimensional grading.) All previous notions such as split complexes, width, and so on, apply equally well to this more general class of complexes; as we will see in Section~\ref{sec:4.1}, there will be no trouble in defining the doubling operation for these complexes also. If $X$ has skeleton $C_{cell}$ with generators $e_i$, we may consider the dual complex $C_{cell}^*$ spanned by the obvious generators $e_i^*$, where $e_b^*$ appears in $\partial^*e_a^*$ if and only if $e_a$ appears in $\partial e_b$. This is equipped with the grading function $\gr(e_i^*) = - \gr(e_i)$, and the dimensional grading is defined so that $e_i^*$ has dimensional grading $-i$. The usual construction then exhibits $X^*$ as a geometric complex with skeleton $C_{cell}^*$.

To reconcile this with the geometric picture, observe that if we shift the dimensional grading up by $n$ while shifting $\gr$ down by $n$, the Maslov grading is preserved. Hence enacting complementary grading shifts does not change the complex $X$. We could thus have equally well defined $\gr(e_i^*) = - \gr(e_i) - n$ and set the dimensional grading of $e_i^*$ to be $n - i$, where $n$ is the dimension of the original skeleton $C_{cell}$. If $C_{cell}$ is an actual cell complex, then the complex $C_{cell}^*$ with this (now non-negative) dimensional grading may be represented by the dual cell complex of $C_{cell}$ in the geometric sense, as in Figure~\ref{fig:10}. Dualizing Examples~\ref{examp:3.2} and \ref{examp:3.3} yield geometric complexes whose skeleta are given by the relative cell complexes of Figure~\ref{fig:10}. See \cite[Remark 5.3]{DS}.

\begin{figure}[h!]
\center
\includegraphics[scale=0.9]{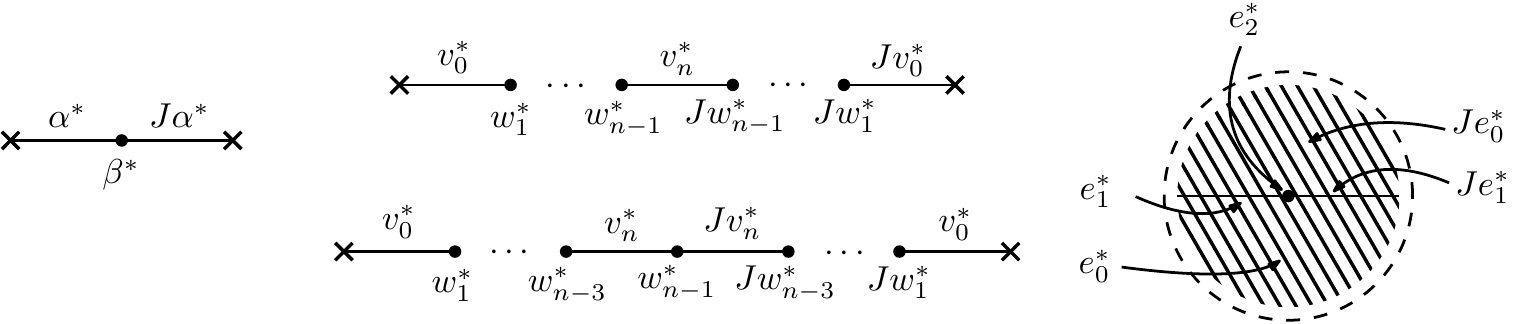}
\caption{The skeleton of $X_i^*$ (left); duals of Example~\ref{examp:3.2} (middle); dual of Example~\ref{examp:3.3} (right). Here, crosses or dashed lines indicate the ``relative" part of the cell complex. For example, on the right-hand side, we have $\partial^* e_0^* = e_1^* + Je_1^*$.}\label{fig:10}
\end{figure}

The dual of a geometric complex is thus also a geometric complex, with the same width and the same number of generators. If the original complex is split, then its dual is split also.

While not absolutely necessary for the results stated in Section~\ref{sec:1}, it is perhaps useful for us to describe the operation dual to doubling. This corresponds to tensoring with $X_\Delta^*$ (rather than $X_{\Delta}$). Consider the cell complexes displayed in Figure~\ref{fig:11}. On the left, we have displayed the skeleta of several split complexes, while on the right we have displayed a corresponding sequence of complexes obtained in each case by ``halving" the $J$-invariant cell on the left. This consists of replacing the old $J$-invariant cell $\eta^*$ with the sum of a new pair of symmetric cells (denoted by $\omega^*$ and $J\omega^*$), which meet along a new $J$-invariant cell $\theta^*$ of one lower dimension.

\begin{figure}[h!]
\center
\includegraphics[scale=0.9]{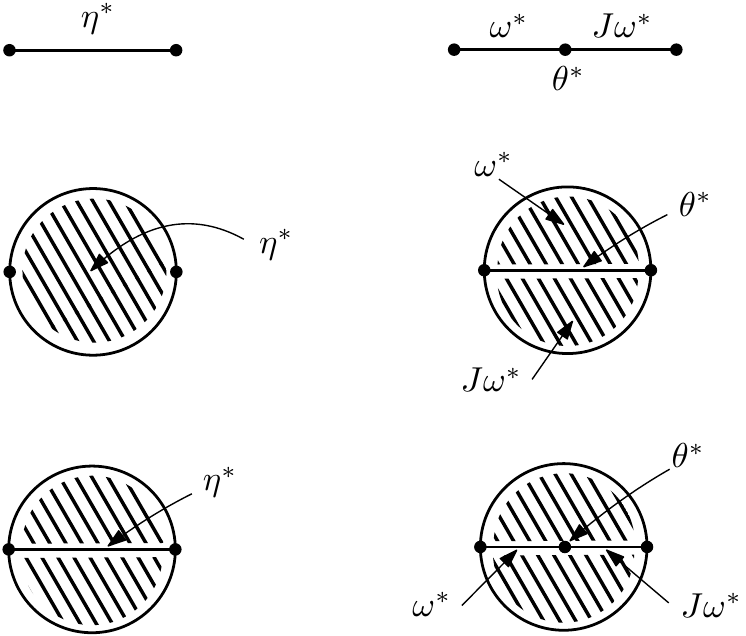}
\caption{Halvings of split complexes.}\label{fig:11}
\end{figure}

There is a slight additional subtlety in the case that $\eta^*$ is a 0-cell. In this situation, we must first take the product of our complex with the one-dimensional complex displayed on the left in Figure~\ref{fig:12}. This has one generator, which we think of as a 1-cell instead of a 0-cell. The value of $\gr$ on this generator is defined to be $-1$, so that its Maslov grading is zero. Taking the product thus algebraically does not change our chain complex, but boosting the dimension by one in the geometric picture allows us to ``half" the $J$-invariant cell, which is now one-dimensional. (We simply shift the dimensional grading up by one and decrease $\gr$ by one, as described previously.) See Figure~\ref{fig:12}.

\begin{figure}[h!]
\center
\includegraphics[scale=0.9]{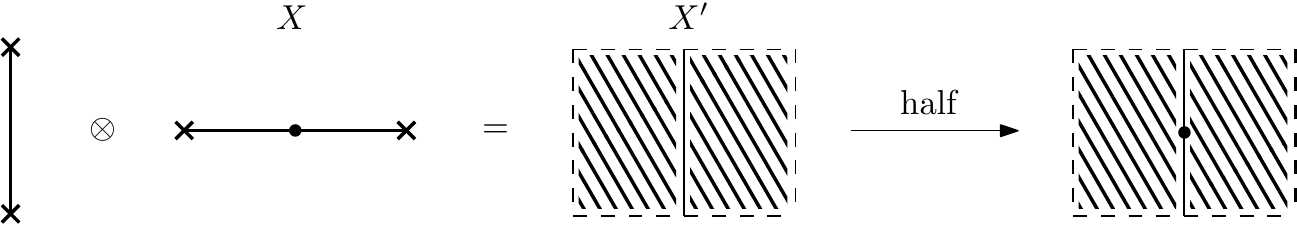}
\caption{Halving a complex $X$ with a $J$-invariant 0-cell. First, we cross $X$ with the complex on the left to obtain an algebraically equivalent complex $X'$, but which we view as having increased dimension. Note that in this example, the result of halving gives a complex equivalent to the one displayed on the right in Figure~\ref{fig:10}.}\label{fig:12}
\end{figure}

Now let $X$ be any split complex, and let $\Delta$ be any non-negative integer with $2\Delta \leq \width(X)$. The \textit{half} of $X$ (with parameter $\Delta$), denoted by $X^h(\Delta)$, is constructed by applying the above halving operation to the skeleton of $X$ and defining $\gr$ as follows:
\begin{enumerate}
\item On every cell except $\omega^*$, $J\omega^*$, and $\theta^*$, the function $\gr$ is equal to its value on the corresponding cell from $X$;
\item $\gr(\omega^*) = \gr(J\omega^*) = \gr(\eta^*)$; and,
\item $\gr(\theta^*) = \gr(\eta^*) + 2\Delta$.
\end{enumerate}
Again, the reader may legitimately worry that we have not rigorously defined the halving operation on skeleta. For now, it is perhaps better to take this as a motivational picture and verify that the halving operation is well-defined and produces the desired complexes in the case of Example~\ref{examp:3.2} and (the dual of) Example~\ref{examp:3.3}.


\section{Split Complexes and Doubling} \label{sec:4}
In this section, we give a precise construction of the doubling operation. We then prove that for any split $\inv$-complex $X$, the doubled complex $X^d(\Delta)$ is locally equivalent to $X \otimes X_{\Delta}$.

\subsection{Doubling Operations, Revisited} \label{sec:4.1}
Let $X$ be a split geometric complex with skeleton $C_{cell}$. Recall this means that $C_{cell}$ has exactly one $J$-invariant cell $\eta$, and all the other cells of $C_{cell}$ occur in $J$-symmetric pairs. We partition the cells of $C_{cell}$ by choosing one cell from each such pair and denoting the span of these cells by $C$. Then $C_{cell}$ (as a module over $\ff$) decomposes as a direct sum
\[
C_{cell} = C \oplus JC \oplus \eta,
\]
where we have written $\eta$ to mean the (single-element) submodule spanned by $\eta$. Note that this splitting is not preserved by $\partial$. Usually when we discuss split complexes we will have in mind a particular splitting obtained by choosing some $C$.

It will be useful to also consider the related splitting $C \oplus (1+ J)C \oplus \eta$, in which case we can write the differential of any $x \in C$ uniquely as
\begin{align*}
\partial x = 
\begin{cases} 
      a + (1 + J)b, \text{ or} \\
      a + (1 + J)b + \eta
\end{cases}
\end{align*}
for $a, b \in C$. We will sometimes use parentheses to denote the possibility of adding $\eta$; i.e., $\partial x = a + (1 + J)b \ (+ \eta)$. Similarly, the fact that $J$ commutes with $\partial$ and that $J\eta = \eta$ implies that
\[
\partial \eta = (1 + J) \zeta
\]
for some $\zeta \in C$. Note that $a$, $b$, and $\zeta$ are allowed to be sums of cells and/or zero. We will continue to use the notation $\partial x = a + (1 + J)b \ (+ \eta)$ throughout the rest of this section.

We now make a proper definition of $X^d(\Delta)$. We begin with the skeleton of $X^d(\Delta)$ as a set:

\begin{defn}\label{defn:4.1}
Let $X$ be a split complex with a particular choice of splitting $C$, and let $\Delta$ be a non-negative integer with $2\Delta \leq \width(X)$. As a module over $\ff$, the skeleton of $X^d(\Delta)$ is defined to be the direct sum
\[
(C \oplus \omega) \oplus (JC \oplus J\omega) \oplus \theta.
\]
Here, $C$ and $JC$ are copies (as modules over $\ff$) of their corresponding summands in $C_{cell}$, while $\omega$, $J\omega$, and $\theta$ are new generators in the skeleton of $X^d(\Delta)$. (We have again abused notation and written individual generators to mean their span over $\ff$.) The action of $J$ interchanges $C$ and $JC$ as before, interchanges $\omega$ and $J\omega$, and sends $\theta$ to itself. Note that this new skeleton is also split, with a choice of splitting given by $(C \oplus \omega)$.
\end{defn}

The differential on $X^d(\Delta)$ is defined as follows:

\begin{defn}\label{defn:4.2}
The differential $\partial^d$ on the skeleton of $X^d(\Delta)$ is defined by modifying the original differential $\partial$ on the skeleton of $X$ to take into account the new cells. We have the following casework:
\begin{enumerate}
\item Let $x \in C$. Then there are three cases: 
\[
\partial^dx = 
\begin{cases} 
      a + (1 + J)b, \text{ if }\partial x = a + (1 + J)b \text{ and } \partial b \text{ does not contain } \eta; \\
      a + (1 + J)b + \theta, \text{ if }\partial x = a + (1 + J)b \text{ and } \partial b \text{ does contain } \eta; \\
      a + (1 + J)b + \omega, \text{ if }\partial x = a + (1 + J)b + \eta. 
\end{cases}
\]
\item $\partial^d\omega = \partial \eta = (1+ J) \zeta$, and 
\item $\partial^d\theta = (1 + J)\omega$.
\end{enumerate}
The differential on $JC$ and $J\omega$ is defined by requiring $J$-equivariance.
\end{defn}

The three cases for $\partial^dx$ above correspond to different situations as in Figure~\ref{fig:8}. The first case is the most straightforward, where $x$ is either sufficiently ``far away" from $\eta$ or is of very different dimension, and the boundary map on $x$ does not need to be modified. (We will sometimes write this as $\partial^dx = \partial x$.) The second case corresponds to a situation in which the boundary of $x$ must be modified to include the new $J$-invariant cell $\theta$. (Consider the 2-cell in the fourth example of Figure~\ref{fig:8}.) Finally, the third case corresponds to the situation where the original boundary of $x$ contains $\eta$, and $\eta$ must be replaced by the corresponding cell $\omega$. (Consider the 2-cell in the third example of Figure~\ref{fig:8}.)

We now verify that $\partial^d$ is a differential (squares to zero). The reader who already believes that the geometric doubling operation is well-defined (and given by the above algebraic construction) may skip this verification, but should be cautioned that the casework below will turn out to be useful in later sections.

\begin{lem}
The differential $\partial^d$ turns $X^d(\Delta)$ into a chain complex; i.e., it squares to zero.
\end{lem}
\begin{proof}
We proceed using copious amounts of casework:
\begin{enumerate}
\item Suppose $x \in C$. We first divide into two cases based on whether or not $\partial x$ contains $\eta$:
\begin{enumerate}
\item Suppose $\partial x$ does not contain $\eta$, so that $\partial x = a + (1 + J)b$. In order to understand $(\partial^d)^2 x$, let us consider the possible forms of $\partial a$ and $\partial b$. Write $\partial a = p + (1 + J)q \ (+ \eta)$ and $\partial b = r + (1 + J)s \ (+ \eta)$, with $p, q, r, s \in C$. Here, we have written $(+ \eta)$ to keep open the possibility of $\partial a$ and/or $\partial b$ containing $\eta$. Since we know $\partial^2x = 0$, however, we immediately see that $\partial a$ cannot contain $\eta$ and that $p = 0$ and $q = r$. Thus we have
\begin{align*}
\partial a &= (1 + J)r \\
\partial b &= r + (1 + J)s \ (+ \eta).
\end{align*}
Now, the fact that $\partial^2 b = 0$ implies that $\partial r$ cannot contain $\eta$. Hence (referring to the definition of $\partial^d$) we have $\partial^d a = \partial a$. We further subdivide into three remaining subcases:
\begin{enumerate}
\item Suppose $\partial b$ does not contain $\eta$, and $\partial s$ does not contain $\eta$. Then $\partial^d x = \partial x$ and $\partial^d b = \partial b$. In this case we obviously have $(\partial^d)^2 x = \partial^2 x = 0$.
\item Suppose $\partial b$ does not contain $\eta$, but $\partial s$ does contain $\eta$. Then $\partial^d b = \partial b + \theta$, and
\begin{align*}
(\partial^d)^2 x &= \partial a + (1 + J)(\partial b + \theta) \\
&= \partial^2x + (1 + J)\theta,
\end{align*}
which is zero, since $J\theta = \theta$.
\item Suppose $\partial b$ does contain $\eta$. Then we have
\begin{align*}
\partial^d x &=  a + (1 + J)b + \theta \\
\partial^d b &= r + (1 + J)s + \omega. 
\end{align*}
Hence
\begin{align*}
(\partial^d)^2 x &= \partial^d( a + (1 + J)b + \theta) \\
&= (1 + J)r + (1 + J)(r + (1 + J)s + \omega) + \partial^d\theta \\ 
&= (1 + J)\omega + \partial^d\theta,
\end{align*}
which is zero, by definition of $\partial^d\theta$.
\end{enumerate}
\item Now suppose $\partial x$ does contain $\eta$, so that $\partial x = a + (1 + J)b + \eta$. Then 
\[
\partial^d x = a + (1 + J)b + \omega.
\]
Due to the dimensional grading, neither $a$ nor $b$ can contain $\eta$ in their differential, and similarly for any cell appearing in $\partial a$ or $\partial b$. Hence $\partial^d a = \partial a$ and $\partial^d b = \partial b$. Since $\partial^d \omega = \partial \eta$, it easily follows that $(\partial^d)^2x = \partial^2 x = 0$.
\end{enumerate}
\item Due to the dimensional grading, $\zeta$ cannot contain $\eta$ in its differential, and similarly for any cell appearing in $\partial \zeta$. Hence $\partial^d \zeta = \partial \zeta$. It follows that $(\partial^d)^2 \omega = \partial^2 \eta = 0$.
\item We evidently have that $(\partial^d)^2 \theta = (1 + J)^2 \zeta = 0$.
\end{enumerate}
By $J$-equivariance, we have that $\partial^d$ squares to zero on $JC$ and $J\omega$ also. This completes the verification that $\partial^d$ is indeed a differential.
\end{proof}

To complete our definition of $X^d(\Delta)$, we finally define the grading function $\gr$ on $X^d(\Delta)$:
\begin{defn}\label{defn:4.3}
We set:
\begin{enumerate}
\item If $x \in C$ (or $JC$), then $\gr(x)$ is the same in $X^d(\Delta)$ as it is in $X$;
\item $\gr(\omega) = \gr(J\omega) = \gr(\eta)$; and
\item $\gr(\theta) = \gr(\eta) - 2\Delta$.
\end{enumerate}
Note that the grading of $\theta$ is the only place where we use the parameter $\Delta$. We remind the reader that the Maslov grading is equal to the sum of the dimensional grading and $\gr$. The dimensions of the cells in $C$ and $JC$ are the same as before, while the dimension of $\omega$ is equal to that of $\eta$, and the dimension of $\theta$ is equal to the dimension of $\eta$ plus one. (Compare Section~\ref{sec:3.2}.)
\end{defn}

In order to verify that this data forms a well-defined geometric complex, we must check the usual monotonicity condition on $\gr$ with respect to $\partial^d$. In all situations except the second case of (1) in Definition~\ref{defn:4.2}, this is either immediate or it follows from the fact that the original differential $\partial$ has this property. In the remaining case, it follows from the fact that $\gr(\theta) = \gr(\eta) - 2\Delta$ and $2\Delta \leq \width(X)$. More precisely, in this situation we know that some summand $b_0$ of $b$ has $\eta$ appearing in its differential. Hence $\gr(x) \leq \gr(b_0) \leq \gr(\eta)$. The condition on $\Delta$ ensures that $\gr(\theta)$ lies between $\gr(b_0)$ and $\gr(\eta)$, so in particular $\gr(x) \leq \gr(\theta)$, as desired.

This completes the verification that $X^d(\Delta)$ is a well-defined geometric complex. Note that $X^d(\Delta)$ is obviously split, and has width $2\Delta$. We think of the skeleton of $X^d(\Delta)$ as being obtained from the skeleton of $X$ via the doubling operation described above, although the reader who is uncomfortable with this point of view can of course rely solely on the algebraic definition.

\subsection{Local Equivalence}\label{sec:4.2}
Let $X$ be a split complex with a choice of splitting given by $C$, and let $X^d(\Delta)$ be its double as defined in Section~\ref{sec:4.1}. The main result of this section will be to show that $X^d(\Delta)$ is locally equivalent to $X \otimes X_{\Delta}$.

We employ the following principle from \cite[Lemma 3.3]{DM}. Let $C_1$ and $C_2$ be two geometric complexes graded by the same coset of $2\ZZ$ in $\mathbb{Q}$, and let $f$ be a cellular chain map from the skeleton of $C_1$ to the skeleton of $C_2$. Suppose that for any pair of cells $x$ from $C_1$ and $y$ from $C_2$ with $y$ appearing in $f(x)$, we have $\gr(y) \geq \gr(x)$. Then we can lift $f$ to a grading-preserving chain map from $C_1$ to $C_2$ by multiplying through by the appropriate powers of $U$; i.e., 
\[
\widetilde{f}(x) = \sum_{y_i \text{ appearing in } f(x)} U^{(\gr(y_i) - \gr(x))/2} y_i.
\]
To construct maps between $X^d(\Delta)$ and $X \otimes X_{\Delta}$, we will produce chain maps between their skeleta and check the monotonicity condition on their gradings separately.

Recall that $X_{\Delta}$ has three generators over $\ff[U]$, which we denote by $\alpha, J\alpha,$ and $\beta$. These have gradings given by $\gr(\alpha) = \gr(J\alpha) = 0$ and $\gr(\beta) = -2\Delta$; the skeleton of $X_{\Delta}$ has differential $\partial \beta = \alpha + J\alpha$.

We begin by constructing a chain map from $X^d(\Delta)$ to $X \otimes X_{\Delta}$. Here, for convenience we use $\cdot$ to denote products of cells; e.g., $x \cdot \alpha$ for the product of $x$ in $X$ and $\alpha$ in $X_{\Delta}$. It may be helpful for the reader to consult Example~\ref{examp:4.4} and Figure~\ref{fig:13}.

\begin{lem}\label{lem:4.1}
We construct a chain map $f$ from $X^d(\Delta)$ to $X \otimes X_{\Delta}$.
\end{lem}
\begin{proof}
We define our map $f$ on generators $x \in C$, $\omega$, and $\theta$ of $X^d(\Delta)$, and extend by requiring $J$-equivariance.
\begin{enumerate}
\item Let $x \in C$, and suppose $\partial x = a + (1 + J)b \ (+ \eta)$. We define
\[
f(x) = x\cdot \alpha + (Jb) \cdot \beta.
\]
\item We define
\[
f(\omega) = \eta \cdot \alpha + (J\zeta)\cdot \beta.
\]
\item We define
\[
f(\theta) = \eta \cdot \beta.
\]
\end{enumerate}
Let us first verify the monotonicity condition on the gradings. If $x \in C$, we see that the first term appearing in $f(x)$ has the same grading as $x$, since $\gr(x \cdot \alpha) = \gr(x)$. All the cells appearing in $b$ have grading greater than or equal to that of $x$ (since $b$ appears in $\partial x$), and the gradings of the cells appearing in $(Jb) \cdot \beta$ are precisely these gradings shifted down by $2\Delta$. Since $2\Delta \leq \width(X)$, it is easily follows that the gradings of the cells appearing in $(Jb) \cdot \beta$ are still greater than or equal to that of $x$. The other cases are checked similarly.

We now show that $f$ is a chain map. As in the proof that $\partial^d$ is a differential, we have several cases to consider:
\begin{enumerate}
\item Let $x \in C$. We divide into two cases based on whether or not $\partial x$ contains $\eta$:
\begin{enumerate}
\item Let $\partial x = a + (1 + J)b$. As before, we have
\begin{align*}
\partial a &= (1 + J)r \\
\partial b &= r + (1 + J)s \ (+ \eta).
\end{align*}
If $\partial b$ does not contain $\eta$, then we compute 
\begin{align*}
f(\partial^dx) &= f(a + (1 + J)b) \\
&= f(a) + f(b) + Jf(b) \\
&= a \cdot \alpha + (Jr) \cdot \beta + b \cdot \alpha + (Js) \cdot \beta + (Jb) \cdot (J \alpha) + s \cdot \beta \\
&= (a + b) \cdot \alpha + (Jb) \cdot (J\alpha) + (Jr + Js + s) \cdot \beta
\end{align*}
and
\begin{align*}
\partial f(x) &= \partial(x \cdot \alpha + (Jb) \cdot \beta) \\
&= (\partial x) \cdot \alpha + J(\partial b) \cdot \beta + (Jb) \cdot (\partial \beta) \\
&= (a + (1 + J)b) \cdot \alpha + (Jr + (1 + J)s) \cdot \beta + (Jb) \cdot (\alpha + J\alpha) \\
&= (a + b) \cdot \alpha + (Jb) \cdot (J\alpha) + (Jr + Js + s) \cdot \beta,
\end{align*}
as desired. If $\partial b$ does contain $\eta$, then the above computation must be slightly modified, and we have an extra term:
\begin{align*}
f(\partial^dx) &= f(a + (1 + J)b + \theta) \\
&= (a + b) \cdot \alpha + (Jb) \cdot (J\alpha) + (Jr + Js + s) \cdot \beta + f(\theta) \\
&= (a + b) \cdot \alpha + (Jb) \cdot (J\alpha) + (Jr + Js + s) \cdot \beta + \eta \cdot \beta.
\end{align*}
Similarly,
\begin{align*}
\partial f(x) &= \partial(x \cdot \alpha + (Jb) \cdot \beta) \\
&= (a + (1 + J)b) \cdot \alpha + (Jr + (1 + J)s + \eta) \cdot \beta + (Jb) \cdot (\alpha + J\alpha) \\
&= (a + b) \cdot \alpha + (Jb) \cdot (J\alpha) + (Jr + Js + s) \cdot \beta + \eta \cdot \beta,
\end{align*}
as desired.
\item Let $\partial x = a + (1 + J)b + \eta$. We again write $\partial a = p + (1 + J)q$ and $\partial b = r + (1 + J)s$, noting that $\partial a$ and $\partial b$ cannot contain $\eta$ due to the dimensional grading. Imposing the condition that $\partial^2x = 0$ yields $p = 0$ and $q + r + \zeta = 0$. We thus compute:
\begin{align*}
f(\partial^dx) &= f(a + (1 + J)b + \omega) \\
&= f(a) + f(b) + Jf(b) + f(\omega) \\
&= a \cdot \alpha + (Jq) \cdot \beta + b \cdot \alpha + (Js) \cdot \beta + (Jb) \cdot (J \alpha) + s \cdot \beta + \\
& \ \ \ \ \eta \cdot \alpha + (J\zeta) \cdot \beta \\
&= (a + b + \eta) \cdot \alpha + (Jb) \cdot (J\alpha) + (Jq + J\zeta + Js + s) \cdot \beta
\end{align*}
and
\begin{align*}
\partial f(x) &= \partial(x \cdot \alpha + (Jb) \cdot \beta) \\
&= (a + (1 + J)b + \eta) \cdot \alpha + (Jr + (1 + J)s) \cdot \beta + (Jb) \cdot (\alpha + J\alpha) \\
&= (a + b + \eta) \cdot \alpha + (Jb) \cdot (J\alpha) + (Jr + Js + s) \cdot \beta,
\end{align*}
with equality following from the fact that $Jq + J\zeta = Jr$, since $q + r + \zeta = 0$.
\end{enumerate}
\item The proof that $f$ is a chain map on $\omega$ is handled similarly. Let $\partial \zeta = r + (1 + J)s$. Imposing the condition that $\partial^2 \eta = 0$, we see that $r = 0$. Hence we compute
\begin{align*}
f(\partial^d\omega) &= f((1 + J)\zeta) \\
&= \zeta \cdot \alpha + (Js) \cdot \beta + (J\zeta) \cdot (J\alpha) + s \cdot \beta
\end{align*}
and
\begin{align*}
\partial f(\omega) &= \partial(\eta \cdot \alpha + (J\zeta) \cdot \beta) \\
&= (\zeta + J\zeta) \cdot \alpha + (1+J)s \cdot \beta + (J\zeta) \cdot (\alpha + J \alpha) \\
&= \zeta \cdot \alpha + s \cdot \beta+ (Js) \cdot \beta + (J\zeta) \cdot (J\alpha),
\end{align*}
as desired.
\item Finally we verify that $f$ is a chain map on $\theta$. We have
\[
f(\partial^d\theta) = f((1 + J)\omega) = \eta \cdot \alpha + (J\zeta) \cdot \beta + \eta \cdot (J\alpha) + \zeta \cdot \beta
\]
and
\[
\partial f(\theta) = \partial(\eta \cdot \beta) = (\zeta + J\zeta) \cdot \beta + \eta \cdot (\alpha + J\alpha),
\]
which are evidently equal.
\end{enumerate}
This completes the verification that $f$ is a chain map.
\end{proof}

We now construct a chain map $g$ from $X \otimes X_{\Delta}$ to $X^d(\Delta)$. It may be useful to consult Example~\ref{examp:4.4} and Figure~\ref{fig:13}.
\begin{lem} \label{lem:4.3}
We construct a chain map $g$ from $X \otimes X_{\Delta}$ to $X^d(\Delta)$.
\end{lem}
\begin{proof}
As usual, we define $g$ piecewise on different generators of $X \otimes X_{\Delta}$.
\begin{enumerate}
\item We first define $g$ on $C \otimes \alpha$. Let $x \in C$. We define
\[
g(x \cdot \alpha) = x,
\]
where on the right-hand side we mean the corresponding element $x \in (C \oplus \omega)$ in $X^d(\Delta)$. By imposing $J$-equivariance, this defines $g$ on $(JC) \otimes (J\alpha)$. Defining $g$ on $C \otimes (J\alpha)$ is similar, but with an extra special case:
\[
g(x \cdot (J\alpha)) = 
\begin{cases} 
      x, \text{ if } \partial x \text{ does not contain } \eta \\
      x + \theta, \text{ if } \partial x \text{ contains } \eta.
\end{cases}
\]
By imposing $J$-equivariance, this also defines $g$ on $(JC) \otimes \alpha$. 
\item We define 
\[
g(\eta \cdot \alpha) = \omega.
\]
Note that imposing $J$-equivariance forces $g(\eta \cdot (J\alpha)) = J\omega$.
\item We define $g$ on $C \otimes \beta$ to be zero. By $J$-equivariance, this defines $g$ on $(JC) \otimes \beta$ to also be zero. We set $g(\eta \cdot \beta) = \theta$.
\end{enumerate}
As usual, we must verify that $g$ satisfies the usual property involving the gradings. In all cases except when we send $x \cdot (J\alpha)$ to $x + \theta$, it is evident that $g$ as written preserves $\gr$; whereas in the exceptional case the claim follows from the fact that $\gr(\theta) = \gr(\eta) - 2\Delta$ and $2\Delta \leq \width(X)$.

We now verify that $g$ is a chain map. As usual, we proceed using copious amounts of casework:
\begin{enumerate}
\item We first check that $g$ is a chain map on $C \otimes \alpha$ and $C \otimes (J\alpha)$. 
\begin{enumerate}
\item First suppose that $\partial x$ does not contain $\eta$, so that $\partial x = a + (1 + J)b$. If $\partial b$ does not contain $\eta$, we have
\begin{align*}
g(\partial (x\cdot \alpha)) &= g((a + (1+J)b)\cdot \alpha) \\
&= g(a \cdot \alpha) + g(b \cdot \alpha) + Jg(b \cdot (J\alpha)) \\
&= a + (1 + J)b
\end{align*}
while
\[
\partial^dg(x \cdot \alpha) = \partial^dx = a + (1 + J)b,
\]
as desired. A similar computation holds for $x \cdot (J\alpha)$. (Note that here we have used the fact that $\partial a$ cannot contain $\eta$, since $\partial^2x = 0$.) Now suppose that $\partial b$ does contain $\eta$. Then $g(b \cdot (J\alpha)) = b + \theta$, rather than just $Jb$. Hence in this case
\[
g(\partial (x\cdot \alpha)) = g((a + (1+J)b)\cdot \alpha) = a + (1 + J)b + \theta,
\]
while
\[
\partial^dg(x \cdot \alpha) = \partial^dx = a + (1 + J)b + \theta,
\]
as desired. Again, a similar computation holds for $x \cdot (J\alpha)$.
\item Now suppose $\partial x = a + (1 + J)b + \eta$. Then we compute
\[
g(\partial (x\cdot \alpha)) = g((a + (1+J)b + \eta)\cdot \alpha) = a + (1 + J)b + \omega,
\]
while
\[
\partial^dg(x \cdot \alpha) = \partial^dx = a + (1 + J)b + \omega,
\]
as desired. Here, we have used the fact that $a$ and $b$ cannot have $\eta$ in their differential, due to the dimensional grading. Similarly,
\[
g(\partial (x\cdot (J\alpha))) = g((a + (1+J)b + \eta)\cdot (J\alpha)) = a + (1 + J)b + J\omega,
\]
while
\begin{align*}
\partial^dg(x \cdot (J\alpha)) &= \partial^d(x + \theta) \\
&= a + (1 + J)b + \omega + (\omega + J\omega) \\
&= a + (1 + J)b + J\omega,
\end{align*}
as desired.
\end{enumerate}
\item We check
\[
g(\partial (\eta\cdot \alpha)) = g((1+J)\zeta\cdot \alpha) = (1 + J)\zeta,
\]
while
\[
\partial^dg(\eta \cdot \alpha) = \partial^d\omega = (1 + J)\zeta,
\]
as desired. 
\item For any $x \in C$, we have that $\partial^dg(x \cdot \beta) = 0$. On the other hand,
\[
g(\partial (x\cdot \beta)) = g((\partial x) \cdot \beta + x \cdot (\alpha + J\alpha)).
\]
If $\partial x$ does not contain $\eta$, then $g((\partial x) \cdot \beta) = 0$ and $g(x \cdot \alpha) = g(x \cdot (J\alpha)) = x$, so this is zero. If $\partial x$ does contain $\eta$, then $g((\partial x) \cdot \beta) = \theta$, while $g(x \cdot \alpha) = x$ and $g(x \cdot (J\alpha)) = x + \theta$, so this is again zero. Finally, we easily verify that 
\[
g(\partial(\eta \cdot \beta)) = g((\zeta +J\zeta) \cdot \beta + \eta \cdot (\alpha + J\alpha)) = \omega + J\omega,
\]
while
\[
\partial^dg(\eta \cdot \beta) = \partial^d(\theta) = (1 + J)\omega,
\]
as desired.
\end{enumerate}
This completes the verification that $g$ is a chain map.
\end{proof}

Now suppose that $X$ is a split $\inv$-complex, so that $U^{-1}H_*(X) = \ff[U, U^{-1}]$. Technically, we have not yet shown that the doubled complex $X^d(\Delta)$ is also an $\inv$-complex, in the sense that we have not formally verified $U^{-1}H_*(X^d(\Delta)) = \ff[U, U^{-1}]$. This property is clear in the geometric picture, since doubling a complex does not change the underlying homotopy type of the skeleton. (Hence doubling a cell complex with trivial homology should give a cell complex which still has trivial homology.) We will provide a more concrete proof in the next section when we explicitly describe the homology of $X^d(\Delta)$. For now, let us state:

\begin{thm}\label{thm:4.3}
Let $X$ be a split $\inv$-complex, and let $\Delta$ be a non-negative even integer such that $2\Delta \leq \width(X)$. Then $X^d(\Delta)$ is locally equivalent to $X \otimes X_{\Delta}$.
\end{thm}
\begin{proof}
Observe that $g(f(\omega)) = \omega$, $g(f(\theta)) = \theta$, and $g(f(x)) = x$ for any $x \in C$. Hence the composition 
\[
g \circ f: X^d(\Delta) \rightarrow X \otimes X_{\Delta}
\]
is the identity. Since $X$ and $X_{\Delta}$ are both $\inv$-complexes, their tensor product is an $\inv$-complex, and we have $U^{-1}H_*(X \otimes X_{\Delta}) = \ff[U, U^{-1}]$. Modulo the fact that $U^{-1}H_*(X^d(\Delta)) = \ff[U, U^{-1}]$, it easily follows that $f$ and $g$ must be isomorphisms between $U^{-1}H_*(X^d(\Delta))$ and $U^{-1}H_*(X \otimes X_{\Delta})$.
\end{proof}

Note that although \textit{a priori} the definition of $X^d(\Delta)$ depends on a choice of splitting $C$ for $X$, Theorem~\ref{thm:4.3} does not depend on this auxiliary data.

\begin{examp} \label{examp:4.4}
It is helpful to work through a concrete example of the maps $f$ and $g$. In Figure~\ref{fig:13}, we have displayed a sample split complex $X$, together with its double $X^d(\Delta)$ and the skeleton of the tensor product $X \otimes X_{\Delta}$. We have marked the $J$-invariant cell $\eta$ in purple, and chosen a splitting $C$ given by the span of the red cells. The corresponding cells of $C$ in $X^d(\Delta)$ are also marked in red, while we have marked $\omega$ in purple and $\theta$ in blue. The reader who wishes to gain intuition for the construction of the maps $f$ and $g$ should stare at Figure~\ref{fig:13} and attempt to construct $J$-equivariant chain maps between $X^d(\Delta)$ and $X \otimes X_{\Delta}$ by hand.

In the second row, we have displayed the action of $f$. Here, the red 0-cell and red horizontal 1-cell are sent to their obvious counterparts in $X \otimes X_{\Delta}$. However, the semicircular red 1-cell - whose boundary contains a nonzero component lying in $(1+J)C$ - is sent to the sum of its corresponding 1-cell in $X \otimes X_{\Delta}$, plus an additional red 1-cell running down the length of the cylinder. Similarly, the red 2-cell is sent to the sum of the two red 2-cells indicated on the right. The purple and blue cells are mapped to each other in the obvious way. The reader should verify that extending via $J$-equivariance yields a cellular chain map.

In the third row, we have displayed the action of $g$. Here, the red, green, purple, and blue cells are all mapped to each other in the obvious way, with the exception that the green 1-cell is mapped to the sum of the green and blue 1-cells on the right. All cells (other than the blue 1-cell) which run along the length of the cylinder are sent to zero. Again, the reader should verify that extending via $J$-equivariance yields a cellular chain map.
\begin{figure}[h!]
\center
\includegraphics[scale=0.8]{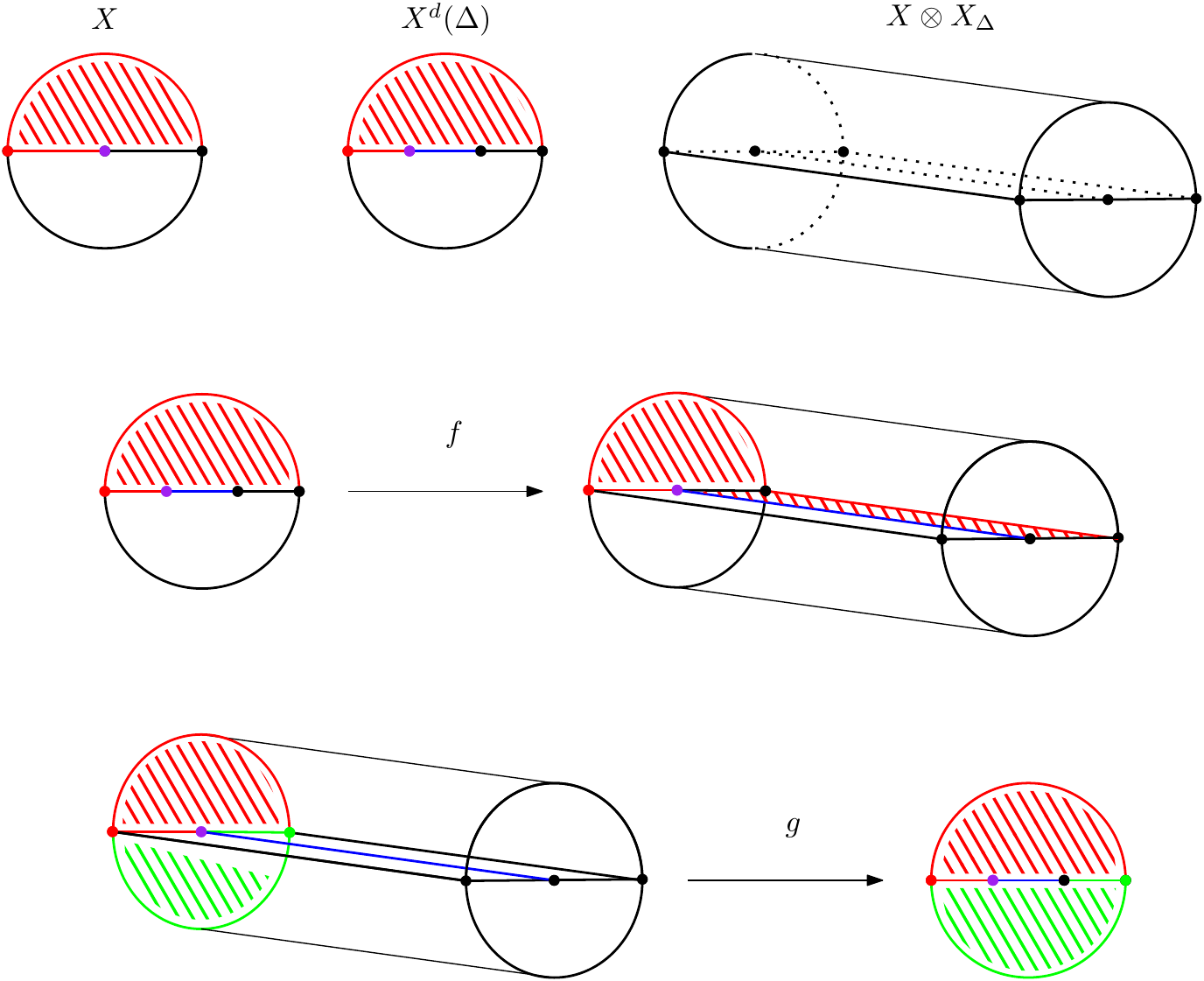}
\caption{Local maps between $X^d(\Delta)$ and $X \otimes X_{\Delta}$.}\label{fig:13}
\end{figure}
\end{examp}


\section{Connected Homology} \label{sec:5}
At this point, it is clear that we can produce a representative complex for any linear combination of the $X_i$ by inductively using Theorem~\ref{thm:4.3} and dualizing whenever needed. More precisely, consider a linear combination
\[
X = \pm X_{i_1} \pm X_{i_2} \pm X_{i_3} \pm \cdots \pm X_{i_n}
\]
with $i_1 \geq i_2 \geq \cdots \geq i_n$. We start with $\pm X_{i_1}$, which is a split complex with width $i_1$. At each step $k$, we either add or subtract $X_{i_k}$. In the former case, we use Theorem~\ref{thm:4.3} (as usual) to obtain a new split complex with width $i_k$. In the latter, we first dualize our original complex to produce a split complex with the same width, and then add $X_{i_k}$ to the dualized complex and dualize again. This similarly yields a new split complex with width $i_k$, and we proceed by induction. (Alternatively, we could explicitly write down the halving operation and show that it is locally equivalent to subtracting $X_{i_k}$.) Note that the representative complex we obtain at the end of this process has $2n + 1$ generators, rather than the $3^n$ generators obtained by naively taking the full tensor product.

We now investigate the homology of the doubled complex. 
\begin{lem} \label{lem:5.1}
Let $X$ be a split complex whose single $J$-invariant cell has Maslov grading $g$. Let $2\Delta \leq \width(X)$, and let $X^d(\Delta)$ be the double of $X$ with parameter $\Delta$. Then 
\[
H_*(X^d(\Delta)) = H_*(X) \oplus \mathcal{T}_g(\Delta).
\]
That is, the homology of $X^d(\Delta)$ is isomorphic to the direct sum of the homology of $X$, plus a single $U$-torsion tower of length $\Delta$ which starts in grading $g$.
\end{lem}
\begin{proof}
Let $X'$ be the subcomplex of $X^d(\Delta)$ spanned over $\ff[U]$ by the two generators $\omega + J\omega$ and $\theta$. Since $\omega + J\omega$ has Maslov grading $g$ and $\partial^d \theta = U^{\Delta}(\omega + J\omega)$, it is clear that the homology of $X'$ is the $U$-torsion tower $\mathcal{T}_g(\Delta)$. Now observe that the quotient complex $X^d(\Delta)/X'$ is isomorphic to $X$ via the map which sends $\omega$ and $J\omega$ (which are identified in the quotient complex) to $\eta$, $\theta$ (which is zero in the quotient complex) to zero, and is the obvious correspondence on all the other generators. We thus have a long exact sequence of $U$-equivariant maps:
\[
\cdots \rightarrow H_*(X') \rightarrow H_*(X^d(\Delta)) \rightarrow H_*(X^d(\Delta)/X') = H_*(X) \rightarrow \cdots.
\]
We begin by showing that the boundary maps in this sequence are zero. To see this, it suffices to prove that the induced inclusion map $i_*: H_*(X') \rightarrow H_*(X^d(\Delta))$ on the level of homology is injective. The non-zero elements of $H_*(X')$ are given by the classes $[U^i(\omega + J\omega)]$ with $0 \leq i < \Delta$. Viewed as elements of $X^d(\Delta)$, the cycles $U^i(\omega + J\omega)$ are still non-zero in homology, since any element in the image of the boundary map $\partial^d$ on $X^d(\Delta)$ must appear with a $U$-power of at least $\Delta$. This shows that we have a short exact sequence of $U$-equivariant maps.

Proving that the sequence splits is also straightforward. We know that the inclusion of $X'$ into $X^d(\Delta)$ induces a $\ff[U]$-equivariant isomorphism of $H_*(X')$ onto its image in $H_*(X^d(\Delta))$. To show that the sequence splits, it suffices to show that this image is in fact a direct summand of $H_*(X^d(\Delta))$. This will be true if the class $[\omega + J\omega]$ in $H_*(X^d(\Delta))$ is not in the image of multiplication by $U$. (In such a situation, it is easy to see that we can extend $[\omega + J\omega]$ to an $\ff[U]$-basis for all of $H_*(X^d(\Delta))$.) But the only way for this to fail would be for $\omega + J\omega$ to be homologous to a $U$-power of some other cycle; i.e.,
\[
\partial x = (\omega + J\omega) + Uy
\]
for some cycle $y$ and bordism element $x$. Because the width of $X^d(\Delta)$ is $\Delta$, this is impossible unless $\Delta = 0$, in which case we have $H_*(X^d(\Delta)) = H_*(X)$ and the theorem is trivially verified.
\end{proof}

Note that Lemma~\ref{lem:5.1} shows that if $U^{-1}H_*(X) = \ff[U, U^{-1}]$, then $U^{-1}H_*(X^d(\Delta)) = \ff[U, U^{-1}]$, completing the proof of Theorem~\ref{thm:4.3}.

Lemma~\ref{lem:5.1} easily implies the following precursor to Theorem~\ref{thm:1.1}.  Below, we still allow cancelling pairs of basis elements (i.e., $X_i - X_i$) to appear in our linear combinations.
\begin{lem}\label{lem:5.2}
Consider a linear combination
\[
X = \pm X_{i_1} \pm X_{i_2} \pm X_{i_3} \pm \cdots \pm X_{i_n}
\]
with $i_1 \geq i_2 \geq \cdots \geq i_n$. Let $\Sigma$ be the representative complex for $X$ obtained by inductively applying Theorem~\ref{thm:4.3} (dualizing when necessary). Then the homology of $\Sigma$ has torsion part given by the algorithm of Theorem~\ref{thm:1.1} (allowing for cancelling pairs $X_i - X_i$). 
\end{lem}
\begin{proof}
We proceed by induction. Let $\Sigma_k$ denote the representative complex for the first $k$ terms in the above linear combination. We simultaneously show that (a) $H_*(\Sigma_k)$ is given by the desired construction, and that (b):
\begin{enumerate}
\item If the sign of $X_{i_k}$ is positive, then the $J$-invariant cell of $\Sigma_k$ has Maslov grading $g_k$ one lower than the tail of the $i_k$-tower.
\item If the sign of $X_{i_k}$ is negative, then the $J$-invariant cell of $\Sigma_k$ has Maslov grading $g_k$ equal to the tail of the $i_k$-tower.
\end{enumerate} 
The statements are easily established in the case of only one summand $\pm X_{i_1}$. Note that dualizing reflects torsion homology towers over the horizontal line of Maslov grading 1/2, so it easily follows that (a) and (b) hold for $\Sigma_k$ if and only if they hold for $- \Sigma_k$. Proceeding inductively, we assume that (a) and (b) hold for $\Sigma_k$ (and also $-\Sigma_k$). Consider $\Sigma_{k+1}$. If the sign of $X_{i_{k+1}}$ is positive, then (a) follows from Lemma~\ref{lem:5.1}, together with (a) and (b) for $\Sigma_k$. We easily compute $g_{k+1}  = g_k - 2i_{k+1} + 1$, since the dimensional grading of the $J$-invariant cell increases by one while the value of $\gr$ decreases by $2i_{k+1}$. This establishes (b). If the sign of $X_{i_{k+1}}$ is negative, then we apply the preceding argument to establish (a) and (b) for  $-\Sigma_{k+1} = -\Sigma_k + X_{i_{k+1}}$. This completes the proof.
\end{proof}
Lemma~\ref{lem:5.2} does not quite give the connected homology, since (for example) any cancelling pair of basis elements $X_i - X_i$ introduces two towers in the homology of $\Sigma$, while in local equivalence $X_i - X_i$ is zero. The content of Theorem~\ref{thm:1.1} is that this is the only further simplification in the homology which can happen. The author would like to thank Jennifer Hom for helpful comments regarding the following proof:

\begin{proof}[Theorem~\ref{thm:1.1}]
We proceed by induction on $n$. For $X = c_1 X_{i_1}$, the claim is a result of Example~\ref{examp:3.3}. Now suppose that we have a maximally simplified linear combination
\[
X = c_1X_{i_1} + c_2X_{i_2} + \cdots + c_nX_{i_n}
\]
with $i_1 > i_2 > \cdots > i_n$, and assume that we have established the claim for
\[
A = c_1 X_{i_1} + c_2 X_{i_2} + \cdots + c_{n-1}X_{i_{n-1}}.
\]
Write $B = c_n X_{i_n}$, so that $X = A \otimes B$. 

We begin by showing that $\Hc(X)$ must have at least $|c_k|$ towers of length $i_k$, for each $1 \leq k \leq n - 1$. To see this, first consider the local equivalence
\[
A = X \otimes (-B).
\]
This shows that $\Hc(A)$ is a summand of $H_*(X \otimes (-B))$, where (without loss of generality) we may replace $X$ with a locally equivalent complex so that $\Hc(X)$ is given by the torsion part of $H_*(X)$. Now, by the inductive hypothesis, $\Hc(A)$ has exactly $|c_k|$ towers of length $i_k$, for each $1 \leq k \leq n - 1$. The torsion homology of $H_*(-B)$ consists of exactly $|c_n|$ towers, each of length $i_n$. Since $i_n$ is not equal to any other $i_k$, it follows from an application of the K\"unneth theorem and an examination of the Tor functor that the only way for $\Hc(A)$ to be a summand of $H_*(X \otimes (-B))$ is for $\Hc(X)$ to contain at least $|c_k|$ towers of length $i_k$, for each $1 \leq k \leq n - 1$.

We now show that $\Hc(X)$ must have at least $|c_n|$ towers of length $i_n$. This follows similarly by observing that we have the local equivalence
\[
B = (-A) \otimes X. 
\]
Indeed, this implies that $\Hc(B)$ is a summand of $H_*((-A) \otimes X)$, where we may replace each factor with a locally equivalent complex whose connected homology is given by its usual homology. Now, $\Hc(B)$ consists of exactly $|c_n|$ towers, each having length $i_n$. By the inductive hypothesis, $\Hc(A)$ does not contain any towers of length $i_n$. It follows from the K\"unneth theorem and an examination of the Tor functor that $\Hc(X)$ must contain at least $|c_n|$ towers of length $i_n$, as desired.

By Lemma~\ref{lem:5.2}, we have a representative complex $\Sigma$ for $X$ whose torsion homology is already given by the construction claimed in the statement of Theorem~\ref{thm:1.1}. Note that $H_*(\Sigma)$ has exactly $|c_k|$ towers of length $i_k$, for each $1 \leq k \leq n$. The previous paragraphs then imply that the connected homology of $X$ must be all of this homology, as desired. The grading shift in the statement of the theorem follows from taking into account the shift by $-2$ in the map $h$, together with the shift by $-d$ in Theorem~\ref{thm:2.5} and the shift by $-1$ in the definition of connected Floer homology.
\end{proof}


\section{Examples and Applications}\label{sec:6}
In this section, we prove Theorem~\ref{thm:1.2}, as well as Corollaries~\ref{cor:1.3} and \ref{cor:1.4}. We begin with a proof of Theorem~\ref{thm:1.2}:
\begin{proof}[Theorem~\ref{thm:1.2}]
Let $l$ and $k$ be positive integers, and let $d$ be any integer. Consider the $\ff[U]$-module
\[
\bigoplus_{i = 0}^{k-1} \mathcal{T}_{d - i(2l - 1)}(l),
\]
formed by concatenating $k$ different $U$-torsion towers of length $l$, where the towers are placed so that the uppermost element of each tower has grading one less than the lowermost element of the previous tower. We refer to this as a \textit{chain of towers} (\textit{of length} $l$). As in Section~\ref{sec:1.1}, we may identify chains of towers as downwards pointing or upwards pointing. In the former case, we define the head of a chain to be the element with maximal grading and the tail to be the element with minimal grading; in the latter, we reverse these conditions. It is clear from the algorithm of Theorem~\ref{thm:1.1} that if we group together all the $U$-torsion towers in $\HF_\mathrm{conn}$ of a fixed length, they must form a chain.

Now sort these chains in descending order according to their tower lengths $l$. According to the algorithm of Theorem~\ref{thm:1.1}, the first chain in this list will either point downwards and have its head in grading  $d(Y) - 1$, or point upwards and have its head in grading $d(Y)$. The orientation of each subsequent chain can then be inductively determined by its position relative to the tail of the previous chain. The theorem easily follows.
\end{proof}

We now re-phrase the computations of \cite{DS} in terms of the connected Floer homology:
\begin{proof}[Corollary~\ref{cor:1.3}]
According to \cite[Theorem 1.2]{DS} as written, we have
\[
h(Y, \s) = (\pm Y_{i_1} \pm Y_{i_2} \pm Y_{i_3} \pm \cdots \pm Y_{i_n})[2\bar{\mu}(Y, \s)],
\]
where $Y_i = X_i[-2i]$ (see the discussion following \cite[Lemma 4.3]{DS}). Our previous decomposition 
\[
h(Y, \s) = (\pm X_{i_1} \pm X_{i_2} \pm X_{i_3} \pm \cdots \pm X_{i_n})[-d(Y, \s)]
\]
thus shows that 
\[
2\bar{\mu}(Y, \s) - 2(\pm i_1 \pm i_2 \pm \cdots \pm i_n) = -d(Y, \s),
\]
where the signs in front of the $i_k$ are the same as the signs appearing in the decomposition of $h(Y, \s)$. Each index $i_k$ contributes a tower of length $i_k$ to $\HF_{\mathrm{conn}}(Y, \s)$, which is downwards pointing if the sign in front of $X_{i_k}$ is positive and upwards pointing of the sign in front of $X_{i_k}$ is negative. The corollary easily follows.
\end{proof}

\begin{proof}[Corollary~\ref{cor:1.4}]
The Rokhlin invariant reduces to the Neumann-Siebenmann invariant modulo 2.
\end{proof}

We now give an illustration of Theorem~\ref{thm:1.1} in which we also give the actual complexes constructed via the doubling and halving operations of Section~\ref{sec:3}.
\begin{examp}
Consider a maximally simplified linear combination of the form
\[
X = X_{i_1} + X_{i_2} - X_{i_3} - X_{i_4} + X_{i_5}.
\]
The representative complex of Lemma~\ref{lem:5.2} may be constructed via a sequence of iterated doubling and halving operations, as displayed on the left in Figure~\ref{fig:14}. We begin with the geometric complex of $X_{i_1}$, double to obtain $X_{i_1} + X_{i_2}$, half to obtain $X_{i_1} + X_{i_2} - X_{i_3}$, and so on. The values of $\gr$ are similarly defined inductively; the final result is displayed in the upper-right in Figure~\ref{fig:14}.

One can verify by hand that the homology of this complex has the form prescribed by Theorem~\ref{thm:1.1}, as follows. The first torsion tower is generated by the sum of the two zero-graded 0-cells, while the second tower is generated by the sum of the two semicircular 1-cells. The other generators are more subtle: the third tower is generated by a closed loop running around either the upper or the lower hemisphere, consisting of the sum of (the appropriate $U$-powers of) four 1-cells. The fourth tower is generated by either of the interior 0-cells, plus the appropriate $U$-power of either of the zero-graded 0-cells. Finally, the fifth tower is generated by the sum of the two interior 0-cells. In general, it is rather tricky to identify the generators of the torsion towers directly in the representative complex, which is why we have instead used Lemma~\ref{lem:5.1} for the proof of Theorem~\ref{thm:1.1}.

\begin{figure}[h!]
\center
\includegraphics[scale=0.9]{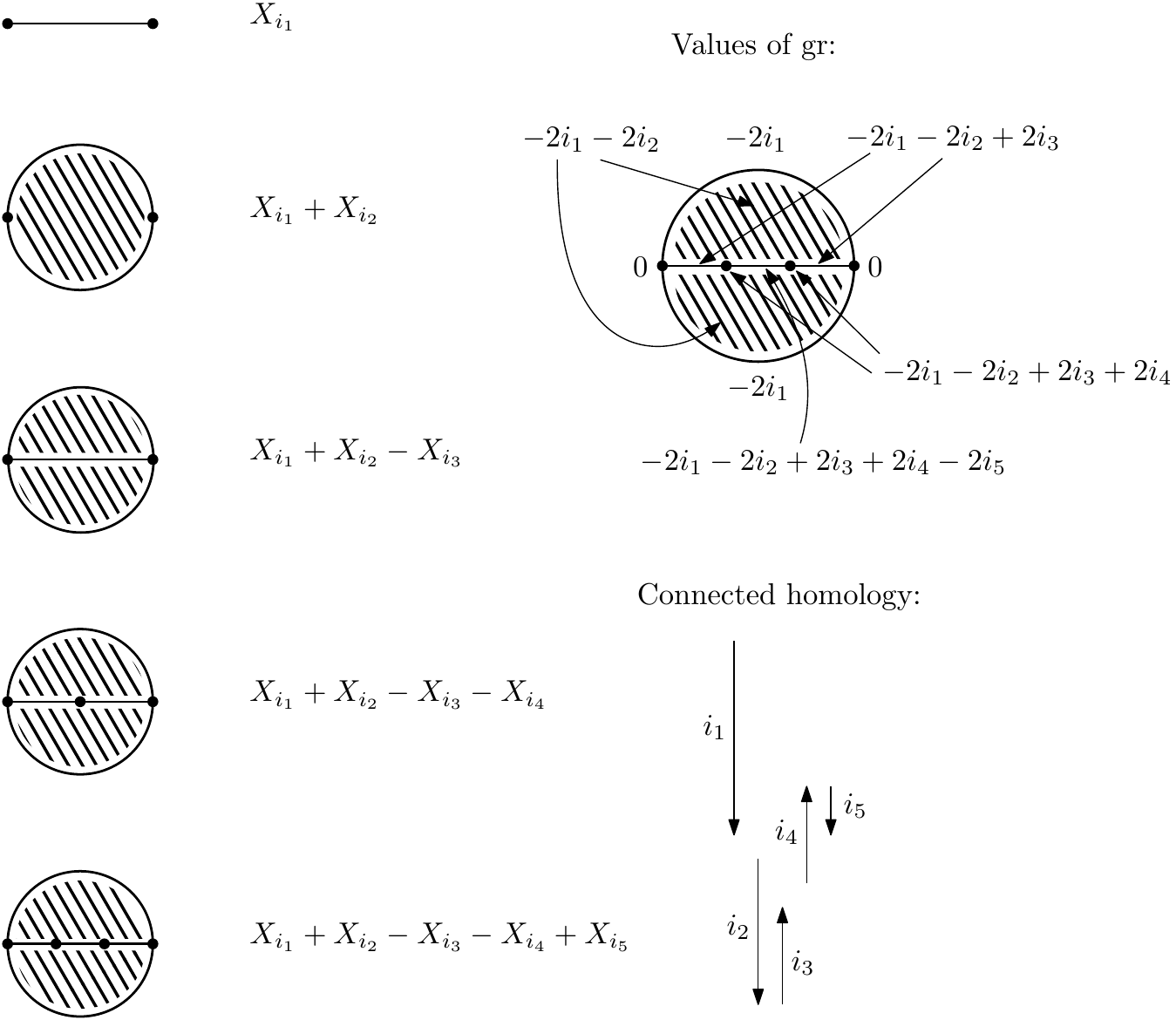}
\caption{Doubling and halving operations (left); final values of $\gr$ (upper-right); homology (lower-right).} \label{fig:14}
\end{figure}
\end{examp}

Finally, we close with an example of an $\inv$-complex whose connected homology is not of the form prescribed by Theorem~\ref{thm:1.1}. Again, we stress that it is not known whether this is realized as the $\inv$-complex of any actual 3-manifold; such an example would show that Seifert fibered spaces do not generate the homology cobordism group.

\begin{examp}\label{ex:6.2}
Consider the geometric complex whose skeleton is given by the usual $J$-equivariant cellular decomposition of $D^2$, as in Figure~\ref{fig:15}. While this is the same skeleton as for $X_{i_1} + X_{i_2}$, we put a ``misordered" grading function on $D^2$, as follows. Fix integers $0 < x < y$. Let the 0-cells have grading zero, the 1-cells have grading $-2x$, and the 2-cell have grading $-2(x + y)$. Explicitly, this means that our $\inv$-complex is generated over $\ff[U]$ by the set $\{e_0, Je_0, e_1, Je_1, e_2 = Je_2\}$, with
\begin{align*}
&\partial e_1 = \partial Je_1 = U^x(e_0 + Je_0) \text{ and} \\
&\partial e_2 = U^y(e_1 + Je_1),
\end{align*}
where $x < y$. The resulting connected homology is displayed (modulo the grading shift by one) on the right in Figure~\ref{fig:15}. This has two $U$-torsion towers of lengths $x$ and $y$, but  the tower of shorter length occurs in higher grading, in contrast to the connected homology of $X_{i_1} + X_{i_2}$. It is easily checked that the connected homology of Figure~\ref{fig:15} (keeping in mind its placement in relation to the $d$-invariant) cannot be achieved through the procedure of Theorem~\ref{thm:1.1}.

\begin{figure}[h!]
\center
\includegraphics[scale=0.9]{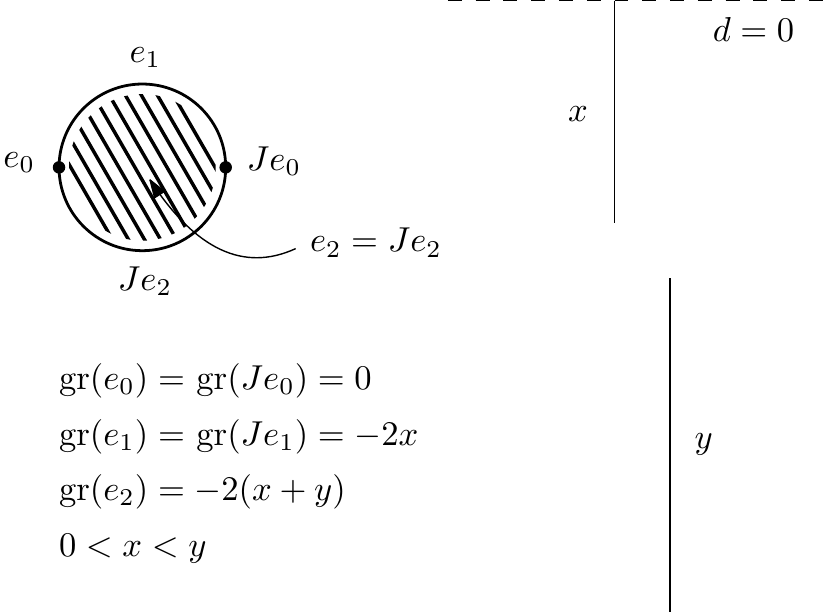}
\caption{A connected homology not arising via the procedure of Theorem~\ref{thm:1.1}.} \label{fig:15}
\end{figure}
\end{examp}

\pagebreak

\end{document}